\documentclass[reqno]{amsart}

\usepackage{amsfonts,color,newclude}
\usepackage{amsthm}
\usepackage{amssymb,transparent}
\usepackage{amsmath}
\usepackage{amscd}
\usepackage{thm-restate}

\usepackage{mathabx}
\usepackage{graphicx}
\usepackage{subfiles}
\usepackage{tikz}
\usepackage{tikz-cd}
\usetikzlibrary{calc}
\usepackage{MnSymbol}
\usepackage[colorlinks=true, pdfstartview=FitV, linkcolor=blue, citecolor=blue, urlcolor=blue]{hyperref}
\usepackage{hyperref}

\usepackage[nameinlink]{cleveref}

\renewcommand{\tilde}{\widetilde}
\newcommand{\mathsym}[1]{{}}

\newcommand{\reals}{\mathbb{R}}

\newcommand{\naturals}{\mathbb{N}}

\newcommand{\integers}{\mathbb{Z}}

\newcommand{\diam}{\operatorname{diam}}

\newcommand{\inv}{^{-1}}

\newcommand{\cyc}[1]{\langle #1 \rangle}

\newcommand{\Z}{\integers}

\newcommand{\oskel}{^{(1)}}

\newcommand{\simpl}{\text{simpl}}

\renewcommand{\bar}{\overline}

\newcommand{\Cay}{\operatorname{Cay}}
\newcommand{\Stab}{\operatorname{Stab }}

\def\mc {\mathcal}

\def\CAT {\ensuremath{\operatorname{CAT}}}

\newcommand{\boundary}{\partial}

\newcommand{\visualBoundary}{\boundary_{\infty}}

\newcommand{\rightQ}[2]{\left.\raisebox{.2em}{$#1$}\middle/\raisebox{-.2em}{$#2$}\right.}
\newcommand{\leftQ}[2]{\left.\raisebox{-.2em}{$#2$}\middle\backslash\raisebox{.2em}{$#1$}\right.}
\newcounter{probnum}
\newtheorem{theorem}{Theorem}[section]
\newtheorem{definition}[theorem]{Definition}

\newtheorem{proposition}[theorem]{Proposition}
\newtheorem{cor}[theorem]{Corollary}
\newtheorem{lemma}[theorem]{Lemma}
\newtheorem{example}[theorem]{Example}

\newtheorem{remark}[theorem]{Remark}

\newtheorem{hypotheses}[theorem]{Hypotheses}
\newtheorem{notation}[theorem]{Notation}

\def\EG {\mc{E}G}
\def\EX {\mc{E}X_0}

\def\sk {\mathrm{sk}}
\renewcommand{\tilde}{\widetilde}
\renewcommand{\hat}{\widehat}
\renewcommand{\epsilon}{\varepsilon}

\title{Relative Cubulation of Small Cancellation Free Products}

\author{Eduard Einstein}
\author{Thomas Ng}

\begin{document}
	\maketitle
	
	\setcounter{tocdepth}{1}
	
	\begin{abstract}
		We expand the class of groups with relatively geometric actions on CAT(0) cube complexes by proving that it is closed under $C'(\frac16)$--small cancellation free products.
		We build upon a result of Martin and Steenbock who prove an analogous result in the more specialized setting of groups acting properly and cocompactly on CAT(0) cube complexes.
		Our methods make use of the same blown-up complex of groups to construct a candidate collection of walls.  
		However, rather than arguing geometrically, we show relative cubulation by appealing to a boundary separation criterion and proving that wall stabilizers form a sufficiently rich family of full relatively quasiconvex codimension-one subgroups. 
		The additional flexibility of relatively geometric actions has surprising new applications.  
		In particular, we prove that $C'(\frac16)$--small cancellation free products of residually finite groups are residually finite.
	\end{abstract}
	
	\section{Introduction}

In is unknown whether (relatively) hyperbolic groups are residually finite (even when the peripheral subgroups are residually finite) \cite[Problem~22]{HitchinProblems},\, \cite[Q~5.H1]{OpenProblemsGGT-BMS},\, \cite[Q~1.15]{BestvinaProblems}, and \cite[Open Problem~1]{ArzhantsevaSteenbock14}.  
On the other hand, work of Agol shows that hyperbolic groups acting properly and cocompactly on CAT(0) cube complexes are residually finite \cite[Theorem~1.1]{AgolVirtualHaken}.  
This article is concerned with the more general context of certain improper actions by relatively hyperbolic groups on CAT(0) cube complexes called \textbf{relatively geometric actions} developed by Einstein and Groves \cite{RelCannon}.  
The additional cubical geometry makes studying residual properties of relatively cubulated groups more tractable than for arbitrary relatively hyperbolic groups \cite{RelGeom, GM20}.  
In particular, we use these ideas to prove the following:

\begin{theorem}\label{T: smcancl res finite}
	Let $A_1, \dotsc, A_n$ be finitely generated residually finite groups with $n \geq 2$.  Any $C'(\frac16)$--small cancellation free product of $A_1, \dotsc, A_n$ by a finite set of relators is residually finite.  
\end{theorem}

We note that \Cref{T: smcancl res finite} has no requirement that the groups $A_i$ are cubulated.  \Cref{T: smcancl res finite} witnesses the versatility of relative cubulations.  

Recently, Martin and Steenbock showed that the class of groups acting properly and cocompactly on CAT(0) cube complexes is closed under taking $C'(\frac16)$--small cancellation free products (see \Cref{D: smcancel fp} for a formal definition) \cite{MartinSteenbock}.  Jankiewicz and Wise gave an alternate proof of the same result in the slightly more restricted setting of $C'(\frac{1}{20})$--small cancellation free products who apply their results to produce examples of cocompactly cubulated groups that do not virtually split as a graph of groups.  

In \cite{RelCannon}, the first author and Groves devised a relatively geometric boundary cubulation criterion based on a relatively hyperbolic (proper and cocompact) cubulation criterion for relatively hyperbolic groups in \cite{BergeronWise}. The criterion in \cite{RelCannon} requires peripherals to be one-ended. Here, we construct relatively geometric actions using \cite[Theorem~1.3]{PeripheralsInfEnds}, a stronger version of \cite[Theorem 2.6]{RelCannon} proved by the authors and Suraj Krishna MS that accommodates arbitrary finitely generated peripherals.  

Our main result is the following generalization of \cite{MartinSteenbock} to relatively geometric actions.

\begin{restatable}{theorem}{mainthm}\label{mainthm}
	Let $(A_i, \mc{P}_i)$ be a collection of $n \geq 2$ finitely generated groups that each act relatively geometrically on $\CAT(0)$ cube complexes, and let $\mc{P} = P_1 \cup \cdots \cup P_n$.  
	If $G$ is a $C'(\frac16)$--free product of $A_1, \dotsc, A_n$ by a finite set of relators, then $(G, \mc{P})$ acts relatively geometrically on a $\CAT(0)$ cube complex.
\end{restatable}

\Cref{mainthm} provides a wealth of new examples of groups with relatively geometric actions on CAT(0) cube complexes. 

Small cancellation theory has roots in Dehn's solution to the word problem for closed hyperbolic surface groups \cite{DehnBook}.
Initial work of Tartakovski\u{\i} \cite{Tartakovskii49}, Greendlinger \cite{Greendlinger1,Greendlinger2}, Lyndon and Schupp \cite{LyndonSchupp}, among others formalized the notion of $C'(\frac16)$ groups, which provides a rich source of word hyperbolic groups.   
A construction of Rips produces important examples of (residually finite) $C'(\frac16)$ groups with exotic quotients \cite{RipsConstruction,WiseRFRips}.  
In particular, small cancellation techniques used in the Rips construction gave rise to examples of torsion-free groups without the unique product property  \cite{GruberMartinSteenbock,RipsSegev,SteenbockRS}.  
Arzhantseva and Steenbock extended these ideas to the setting of $C'(\frac16)$--small cancellation free products \cite{ArzhantsevaSteenbock14}.

$C'(\frac16)$--small cancellation free products generalize classical $C'(\frac16)$ groups by allowing factor groups other than $\Z$.  
Basic examples include one-relator quotients of the form $\rightQ{A*B}{\llangle r \rrangle}$ where $A,B$ are groups and $r$ is a word in $A*B$ satisfying a technical overlap condition discussed in detail in \Cref{S: background}.  
Small cancellation free products have several additional applications.  For example they can be used to solve embedding problems for infinite groups \cite{MillerSchupp,SchuppEmbeddingsSimple} and to construct examples of non-degenerate hyperbolically embedded subgroups of non-relatively hyperbolic groups \cite[Theorem 6.5]{GruberSisto}. 
Most importantly for us, the factors of $C'(\frac16)$--small cancellation free products embed and the resulting group is hyperbolic relative to the factors \cite{Pankratev} (see also \cite[Theorem~1]{SteenbockRS} and \Cref{P: hyp rel vertex} which follows from \cite{LyndonSchupp}).

Residual properties of  hyperbolic groups acting on CAT(0) cube complexes, especially separability of quasiconvex subgroups, played a central role in Agol's resolution of the Virtual Haken Conjecture \cite{AgolVirtualHaken} and the Virtual Fibering Conjecture \cite{AgolVirtualFibering} for closed hyperbolic 3-manifolds.  
Using cubical techniques, Wise proved virtual fibering for finite-volume hyperbolic 3-manifolds \cite{WiseManuscript}. 
There have been many other efforts to extend cubical machinery to relatively hyperbolic groups, e.g. \cite{ArxivThesis,HruskaWise14,HW2015,OregonReyes,SW2015}, but most of these approaches have focused on proper and cocompact or cosparse actions of groups on CAT(0) cube complexes. 

Groups that act on $\CAT(0)$ cube complexes may fail to be residually finite and can even have no non-trivial quotients.  Indeed, there are examples of non-residually finite groups and infinite simple groups that act properly and cocompactly on products of simplicial trees \cite{BurgerMozes, WiseCSC} (see also \cite{CapraceBMWsurvey} and references therein).  
On the other hand, groups acting relatively geometrically on CAT(0) cube complexes have favorable residual properties when the peripheral subgroups are residually finite, see \Cref{T: fqcerf} from \cite{RelGeom} below. \Cref{mainthm} and \Cref{T: fqcerf} together imply:
\begin{cor}\label{res finite}
	Let $G$ be a small cancellation free product as in \Cref{mainthm}.
	If each element of $\mc{P}_i$ is residually finite for all $i$, then $G$ is residually finite and all full relatively quasiconvex subgroups are separable. 
\end{cor} 
\Cref{T: smcancl res finite} is much stronger than similar residual finiteness results that follow from \cite{MartinSteenbock,JankiewiczWise} because the relatively geometric machinery requires fewer restrictions on the factor groups. 
\Cref{T: smcancl res finite} is a special case of \Cref{res finite}, see \Cref{S: relgeom background} for a proof. 

We hope this article also serves as an advertisement for constructing relative cubulations using the boundary cubulation criteria from \cite{RelCannon,PeripheralsInfEnds}. While Martin and Steenbock must explicitly construct a single wall space to cubulate in \cite{MartinSteenbock}, these relatively geometric criteria allow us to build a sufficient collection of codimension-$1$ subgroups for relative cubulation from simultaneously considering two different actions. 


\subsection{Structure of this paper:}

The proof of \Cref{mainthm} for general $C'(\frac16)$--small cancellation free products is not substantially different from the proof for examples of the form $G = \rightQ{A * B}{\llangle w \rrangle}$.   
For ease of notation we restrict ourselves to this special case and remark where additional work is necessary if it is not obvious.  

In \Cref{S: background}, we record some basic facts about relatively geometric actions and small cancellation free products.   In particular, we describe a boundary criterion \Cref{t:Rel BW} for relative cubulation of Krishna MS and the authors, the relationship between relative geometric actions and residual finiteness \Cref{T: fqcerf}, and the proofs of both \Cref{res finite} and \Cref{T: smcancl res finite} assuming the statement of \Cref{mainthm}. 

Our proof of \Cref{mainthm} relies on a blow-up construction for non-positively curved complexes of groups originating in work of Martin \cite{Martin:NPC_boundary} that Martin and Steenbock also employed in the setting of $C'(\frac16)$--small cancellation free products in \cite{MartinSteenbock}.  
We review details of this construction in \Cref{S: MS complexes}.  Specifically, we recall a developable complex of groups structure for $G$ in \Cref{S: complex}, relative hyperbolicity and acylindricity of the action in \Cref{S: relhyp,S: acylindricity}, and the blown-up complex $\EG$ in \Cref{SS:blowupConstruction}.  

In \Cref{S: wall construction}, we introduce \textbf{hyperstructures}  (see \Cref{D:hyperstructure}), and study properties of the hyperstructures on $\EG_{bal}$ and $X_{bal}$, which are subdivisions of $\EG$ and the development of the complex of groups from \Cref{S: MS complexes} respectively. In \cite{MartinSteenbock}, the hyperstructures of $\EG_{bal}$ and lifts of the hyperstructures of $X_{bal}$ to $\EG_{bal}$ were used (with different names) to construct a wallspace on $\EG_{bal}$ that yields a geometric action on a CAT(0) cube complex when the vertex groups are properly and cocompactly cubulated.

Our approach diverges from \cite{MartinSteenbock} at this point. We develop a general framework for relatively cubulating a relatively hyperbolic group acting on a polygonal complex where each of the vertex stabilizers admits a relatively geometric cubulation. We anticipate this framework will also be useful in forthcoming work for relatively cubulating random quotients of free products. 
Since our blowup $\EG_{bal}$ is locally infinite, new techniques are needed to turn the hyperstructures into appropriate codimension--$1$ subgroups. 

Our proof relies heavily on techniques for studying \textbf{generalized fine graphs} (see \Cref{D:GeneralizedFine}) developed by Groves and the authors in \cite{GenFine}, which we review in \Cref{S: Peripheral Structures}. These tools allow us to study Bowditch boundaries using the geometry of the complexes described in \Cref{S: MS complexes}. 
We describe the relationship between the Bowditch boundary of $G$ with respect to the coarser peripheral structure consisting of the entire factor groups and the Bowditch Boundary of $G$ with respect to the peripheral subgroups of the factor groups using results from \cite{WenyuanYang2014}.

In \Cref{S: polygonal case}, we show that if a relatively hyperbolic pair acts on a complex like $X_{bal}$ so that maximal parabolics stabilize exactly one vertex and edge stabilizers are finite, then any pair of points in the Bowditch boundary can be separated by a finite index subgroup of a hyperstructure stabilizer. We also prove that stabilizers of hyperstructures are full and relatively quasiconvex subgroups. 
\Cref{S: X hypotheses satisfied} is devoted to showing that $X_{bal}$ satisfies the hypotheses of the machinery developed in \Cref{S: polygonal case}.

In \Cref{S: blowup main}, we describe relatively geometric blowups of polygonal complexes and prove that they admit relatively geometric cubulations. We show that any two boundary points not in the boundary of a vertex stabilizer of the polygonal complex can be separated using results from \Cref{S: polygonal case}. Then we separate points in the boundary of a single vertex stabilizer by extending a hyperplane from the relatively geometric cubulation of the vertex stabilizer to a full relatively quasiconvex codimension--$1$ subgroup. At this point we will have all the hyppotheses needed to apply \Cref{t:Rel BW}. 
Finally, we conclude the paper with the proof of \Cref{mainthm}: we show that $\EG_{bal}$ is relatively geometric blowup and apply the main result of \Cref{S: blowup main}.  

\subsection*{Acknowledgements}
The authors would like to thank Daniel Groves, Chris Hruska and Jason Manning for useful conversations that helped shape this work.
The authors also thank Michael Hull
for carefully reading a previous version and providing helpful comments and corrections. We also thank Suraj Krishna MS, MurphyKate Montee and Markus Steenbock for conversations and comments that helped correct an error in a prior proof of \Cref{P: onecross}. 
EE was partially supported by an AMS--Simons travel grant and a postdoctoral fellowship from the Swartz Foundation.
TN was partially supported by ISF grant 660/20, an AMS--Simons travel grant, and at the Technion by a Zuckerman Fellowship.
	
	

\section{Background}	\label{S: background}

Our main objects of study are relatively hyperbolic groups.  Relatively hyperbolic groups generalize the class of word hyperbolic groups in the sense of Gromov \cite{GromovHyperbolicGroups} to also include fundamental groups of cusped hyperbolic manifolds with finite volume among many other important examples.  
Heuristically, relatively hyperbolic groups have Cayley graphs that look $\delta$--hyperbolic away from translates of certain subgroups.
One way to formalize this is due to Bowditch and uses \textbf{fine hyperbolic graphs}.

\subsection{Fine hyperbolic graphs and relative hyperbolicity} 
A \textbf{circuit} in a graph $\Gamma$ is an embedded loop. 
Recall that a graph $\Gamma$ is \textbf{fine} if for every edge $e$ and $n\in\naturals$, there exist finitely many circuits that pass through $e$ and have length at most $n$. 
For example, the coned-off Cayley graph of a relatively hyperbolic group pair is a fine hyperbolic graph. 
There are many equivalent definitions of relative hyperbolicity, see for example \cite{Hruska2010}. We will mainly use the following definition for relative hyperbolicity in terms of fine hyperbolic graphs:
\begin{definition}[{\cite[Definition 2]{BowditchRH}, written as stated in \cite[Definition 3.4 (RH-4)]{Hruska2010}}]\label{D: fine relhyp}
	Suppose $G$ acts on a $\delta$-hyperbolic graph $\Gamma$ with finite edge stabilizers and finitely many $G$-orbits of edges. If $K$ is fine, and $\mc{P}$ is a set of representatives of the conjugacy classes of infinite vertex stabilizers, then $(G,\mc{P})$ is a \textbf{relatively hyperbolic (group) pair}.
	Conjugates of subgroups in $\mc{P}$ are called \textbf{peripheral subgroups} or \textbf{maximal parabolic subgroups}.
\end{definition}

In \Cref{S: MS complexes}, we will also briefly use an equivalent definition of relative hyperbolicity given in terms of linear relative isoperimetric function to prove \Cref{P: hyp rel vertex}.

We will need to construct \textbf{full relatively quasiconvex} subgroups of the pair $(G,\mc{P})$ to construct relative cubulations. We use the following criterion to show that a subgroup is relatively quasiconvex:
\begin{theorem}[{\cite[Theorem 1.7]{WiseMP10}}] \label{T: qccrit}
	Let $(G,\mc{P})$ be a relatively hyperbolic pair acting cocompactly on a fine hyperbolic graph so that every edge stabilizer is finite. A subgroup $H\le G$ is relatively quasi-convex in $(G,\mc{P})$ if and only if $H$ stabilizes a quasi-convex $H$--cocompact connected subgraph of $\Gamma$. 
\end{theorem}
We will not need to engage further with the technical details of being relatively quasiconvex. Therefore, we omit a formal definition and refer the interested reader to \cite[Section 6]{Hruska2010}, which gives several equivalent definitions of relative quasiconvexity.

Fullness is required for our cubulation criterion: 
\begin{notation} Let $G$ be a group, $H\le G$ and $g\in G$. Then $H^g = gPg\inv$. \end{notation}

\begin{definition}
	Let $(K,\mc{D})$ be a relatively hyperbolic group pair. A subgroup $H\le K$ is \textbf{full} if for all $D\in\mc{D}$ and $k\in K$, $|H\cap D^k| = \infty $ implies that $H\cap D^k$ is finite index in $D^k$. 
\end{definition}

\subsection{Relatively geometric actions on $\CAT(0)$ cube complexes} 
\label{S: relgeom background} 
The $\CAT(0)$ condition is a metric notion of non-positive curvature.  $\CAT(0)$ cube complexes comprise an important family of examples where the geometry of the space is naturally encoded by the combinatorics of codimension-one subsets called hyperplanes.  For much more on $\CAT(0)$ cube complexes see for example \cite{SageevNotes,WiseBook}.  
Group acting properly and cocompactly on $\CAT(0)$ cube complexes, sometimes referred to as \emph{cubulated groups},  have been well-studied.  
The first author and Groves introduced the following in \cite{RelCannon}. 
\begin{definition} \label{D: relgeom}
	Let $(G, \mc{P})$ be a relatively hyperbolic pair where $G$ acts by isometries on a CAT(0) cube complex $\tilde{X}$. The action of $(G,\mc{P})$ is \textbf{relatively geometric (with respect to $\mc{P}$)} if:
	\begin{enumerate}
		\item the action of $G$ on $\tilde{X}$ is cocompact, 
		\item every peripheral subgroup $P \in \mc{P}$ acts elliptically, and
		\item cell stabilizers are either finite or conjugate to a finite index subgroup of some $P \in \mc{P}$. 
	\end{enumerate}
	Groups that admit a relatively geometric action with respect to some collection of peripheral subgroups are called \textbf{relatively cubulated}.
\end{definition}

Given a subgroup $H\leq G$, we write $\Lambda H$ to denote its limit set.
A relatively hyperbolic group may admit a relatively geometric action with respect to only certain peripheral structures.  
Indeed the action of any group $G$ on a point is a relatively geometric action for $(G, \{G\})$.

\begin{definition}
	Let $(K,\mc{D})$ be a relatively hyperbolic pair. If there is a collection of subgroups $\mc{D}_0$ such that each $D_0\in\mc{D}_0$ is conjugate into some $D\in \mc{D}$ and $(G,\mc{D}_0)$ is relatively hyperbolic, then we say that $(G,\mc{D}_0)$ is a \textbf{refined peripheral structure} for $(G,\mc{D})$.
\end{definition}

We use the following relatively geometric cubulation criterion to produce a relatively geometric action:
\begin{restatable}{theorem}{cubulationcriterion}\label{t:Rel BW} $(${\cite[Theorem 1.3]{PeripheralsInfEnds}, see also  \cite[Theorem 2.6]{RelCannon}}$)$ 
	Let $(G,\mc{P})$ be a finitely generated relatively hyperbolic and suppose that for every pair of distinct points $p$ and $q$ in the Bowditch Boundary $\partial_{\mc{P}}G$ there is a full relatively quasiconvex codimension $1$ subgroup $H$ of $G$ so that $p$ and $q$ lie in $H$--distinct components of $\partial G \smallsetminus \Lambda H$.  Then there exist finitely many full relatively quasiconvex codimension--$1$ subgroups of $G$ and a refined peripheral structure $(G,\mc{P}')$ for $(G,\mc{P})$ so that the action of $(G,\mc{P}')$ on the dual cube complex is relatively geometric.
	
	Moreover, if each $P \in \mc{P}$ acts elliptically on the dual cube complex then no refinement is needed, i.e., we may take $\mc{P}' = \mc{P}$.
\end{restatable}
To prove \Cref{mainthm}, we will provide a sufficient collection of full relatively quasiconvex codimension--$1$ subgroups of a small-cancellation free product that satisfy the hypotheses of \Cref{t:Rel BW}.
Relatively geometric actions have favorable residual properties when the peripherals are residually finite:
\begin{theorem}[{\cite[Corollary 1.7]{RelGeom}, see also \cite[Theorem 4.7]{GM20}}]\label{T: fqcerf}
	Let $(G,\mc{P})$ act relatively geometrically on a CAT(0) cube complex $\tilde{X}$. If every $P\in\mc{P}$ is residually finite, then:
	\begin{enumerate}
		\item $G$ is residually finite, and
		\item every full relatively quasiconvex subgroup of $G$ is separable. 
	\end{enumerate}
\end{theorem}

The statement of \Cref{mainthm} implicitly relies on the fact that the resulting small cancellation free product is hyperbolic relative to the union of the peripheral subgroups of the factor groups, see \Cref{P: hyp rel vertex} and \Cref{rh structure P}.  
With \Cref{T: fqcerf}, \Cref{res finite} follows immediately from \Cref{mainthm}. Assuming \Cref{mainthm}, we can then prove \Cref{T: smcancl res finite}:
\begin{proof}[Proof of \Cref{T: smcancl res finite} assuming \Cref{mainthm}]
	Observe that $(A_i,\{A_i\})$ is a relatively hyperbolic pair, and the trivial action of $(A_i,\{A_i\})$ on a point is a relatively geometric action. Therefore \Cref{mainthm} implies that if $G$ is any $C'(\frac16)$-small cancellation free product of $A_1,A_2,\ldots,A_n$, then $G$ acts relatively geometrically on a CAT(0) cube complex with respect to the peripheral structure $\{A_1,\ldots,A_n\}$. 
	Therefore \Cref{res finite} implies that $G$ is residually finite. 
\end{proof} 

\subsection{Small cancellation free products} 
\label{SS:smallCancel}

Let $A_1, \dotsc, A_n$ be groups.
Let $F = A_1 * \cdots * A_n$ be the free product. 
We will refer to the $A_i$ as \textbf{vertex groups} or \textbf{factor groups}. 
Recall that every non-trival element $w\in F$ can be written uniquely in a normal form $w=w_1w_2w_3\ldots w_\ell$ where each $w_i$ lies in a vertex group and $w_i,w_{i+1}$ are in distinct vertex groups. 
Then there is a natural length $|w|$ where $|w|=\ell$.
The word $w$ is \textbf{cyclically reduced} if whenever $|w|>1$, $w_1\ne w_\ell\inv$. 

A subset $\mc{R}\subseteq F$ is \textbf{symmetrized} if every $r\in\mc{R}$ is cyclically reduced and every cyclically reduced conjugate of $r^{\pm 1}$ lies in $\mc{R}$. 

Let $G = \rightQ{F}{\llangle \mc{R} \rrangle}$. Adding all cyclically reduced conjugates of elements of $\mc{R}$ and their inverses does not change the quotient group $G$, so we can assume $\mc{R}$ is symmetrized whenever it is necessary. 

A reduced word $p\in F$ is a \textbf{piece} with respect to a symmetrized set $\mc{R}$ if there exist $r_1,r_2\in\mc{R}$ (not necessarily distinct) so that 
\begin{enumerate}
	\item there exist $u,v\in F$ so that $r_1 = pu$ and $r_2=pv$,
	\item $u$ and $v$ have no common prefix.
\end{enumerate}
If $\mc{R}$ is not symmetrized, we say that $p$ is a piece with respect to $\mc{R}$ if it is a piece with respect to the smallest symmetrized $\mc{R}_0$ containing $\mc{R}$.

\begin{definition}\label{D: smcancel fp}
	Let $A_1, \dotsc, A_n$ be groups and let $\mc{R}$ be a finite subset of $F=A_1 * \cdots * A_n$. The quotient:
	\[\rightQ{F}{\llangle \mc{R} \rrangle}\]
	is a \textbf{$C'(\frac16)$-free product of the $A_i$ with respect to a finite set of relators $\mc{R}$} if for any piece $p$ and every $r\in \mc{R}$ such that $r= pu$ for some $u\in F$:
	\[|p|<\tfrac16|r|.\] 
	To avoid pathologies, we will also insist that $|r|\ge 6$. 
\end{definition}

We now shift our attention to $C'(\frac16)$--small cancellation free products where the vertex groups are each relatively cubulated.
As mentioned in the introduction, proofs for general small cancellation free products can often be deduced from the case of two vertex groups  and a single relation.  
To this end, we use the following notation.  
\begin{notation}
	Let $A$ and $B$ be relatively hyperbolic groups with peripherals $\mc{P}_A$ and $\mc{P}_B$ respectively.  
	
	Let $EA$ and $EB$ be CAT(0) cube complexes on which the groups $A$ and $B$ admit relatively geometric actions with respect to $\mc{P}_A$ and $\mc{P}_B$ respectively.

	From now on, $G$ is a $C'(\frac16)$ free product of groups $A,B$ with respect to a single relation $R$.
\end{notation}

We will review a construction from \cite{MartinSteenbock} for a complex of groups whose fundamental group is $G$ in \Cref{S: MS complexes}. The main  difference between the case with $2$ vertex groups and the case with $n$ vertex groups is the construction of the appropriate complex of groups, which we will briefly address in the next section.  

\section{A complex of groups for $G$ and the construction of its development}\label{S: MS complexes}

We will make use of a developable complex of groups structure for small cancellation free products due to Martin and Steenbock \cite{MartinSteenbock}.  The purpose of describing this construction is to build an action of $G$ on a 2-dimensional polygonal complex $X$ that captures the relative hyperbolicity of $G$ with respect to the peripheral structure $\mc{Q}$ consisting of factor groups.  
The complex $X$ is also the starting point for building a blown-up space that captures the finer relatively hyperbolic structure $(G, \mc{P})$ coming from peripherals within the factor groups.  

Throughout this section, we illustrate the construction when $G$ is a small cancellation free product of two groups with a single relator for clarity and ease of notation. 
Specifically, we assume that $G=\rightQ{A*B}{ \llangle w^d \rrangle }$ where $w$ is a cyclically reduced word of length $2N$ in $A*B$ that is not a proper power and $d \geq 0$.

\subsection{The developed complex $X$, via complex of groups structure for $G$}
\label{S: complex}

Fundamental groups of developable complexes of groups naturally act on connected and simply connected cell complexes \cite[Theorem~III.$\mc{C}$.3.13]{BridsonHaefliger}.  
As with graphs of groups, developable complexes of groups structure can be understood as decompositions of a group into subgroups that determine (up to conjugacy) stabilizers of cells in the universal cover of the complex.  


Let $L$ be a connected complex consisting of two vertices labeled $u_A$ and $u_B$ with a single edge joining them, and let $L'$ be its barycentric subdivision.

\begin{definition} \label{keeping vertices and edges straight} 
	Let $R_{0,\simpl}$ be a complex constructed as follows:
	\begin{enumerate}
		\item Start with a $2N$-gon, label vertices consecutively (with respect to some orientation on the boundary with labels taken modulo $2N$) $v_0,v_1,\ldots,v_{2N-1}$. 
		\item Add an \textbf{apex vertex} to the center of the $2N$-gon and join the apex vertex to each $v_i$ by adding an edge $e_i$ to obtain a cone over a loop on $2N$ edges. Let $\sigma_i$ denote the triangle that has both $e_{i-1},e_i$ as edges.
	\end{enumerate}
\end{definition}

Construct a new complex $K_\simpl$ by gluing the vertices of $R_{0,\simpl}$  with even subscripts to $u_A$, those vertices with odd subscripts to $u_B$ and gluing the edge between $v_i$ and $v_{i+1}$ to the edge joining $u_A$ to $u_B$ as shown in \Cref{fig:complex}.  We emphasize that $K_\simpl$ is compact because it is a finite cell complex.

\begin{definition}
	The \textbf{scwol} over $K_\simpl$ is a small category $\mc K$ whose objects are the cells of $K_\simpl$. 
	There is an arrow between objects from $\sigma_2$ to $\sigma_1$ whenever $\sigma_1$ is a face of $\sigma_2$ (reverse inclusion). 
\end{definition}
See \cite[Section III.$\mc{C}$.1]{BridsonHaefliger} for a general treatment of scwols. 

\begin{figure}[h]
	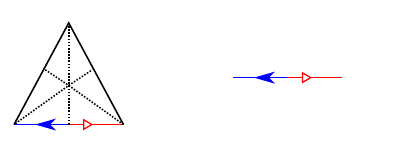
	\caption{The building blocks of $G(\mc{K})$. The underlying complex of $G(\mc{K})$ is built from the disjoint union of the two edge graph $L'$ on the right and a polygon consisting of $2N$ copies of the triangle on the left glued together so that the apex is at the center of the polygon. The polygon and $L'$ are glued together as shown.}
	\label{fig:complex}
\end{figure}

\begin{definition}
	A \textbf{complex of groups} $G(\mc{K})$ over a scwol $\mc{K}$ is given by the following data:
	\begin{enumerate}
		\item for each object $\sigma$ of $\mc{K}$, a group $G_\sigma$ called the \textbf{local group},
		\item for each arrow $f$, an injective homomorphism $\psi_f:G_{i(f)}\to G_{t(f)}$,
		\item if $t(f_2) = i(f_1)$, there is an element $g_{f_1,f_2}\in G_{t(f_1)}$ so that:
		\[\text{Ad}_{g_{f_1,f_2}} \psi_{f_1f_2} = \psi_{f_1} \psi_{f_2}\]
		where $\text{Ad}_g$ is conjugation by $g$. The element $g_{f_1,f_2}$ is called a \textbf{twisting element}.	
	\end{enumerate}
	Additionally, the data must satisfy the \textbf{cocycle condition}: for any triple $f_1,f_2,f_3$ such that $f_1f_2f_3$ is a valid composition of arrows in $\mc{K}$:
	\[\psi_{f_1}(g_{f_2,f_3})g_{f_1,f_2f_3} = g_{f_1,f_2}g_{f_1f_2,f_3}\]
\end{definition}
For much more detail about complexes of groups, see \cite[Chapter III.$\mc{C}$]{BridsonHaefliger}.

Let $\mc{K}$ be the scwol over $K_{\simpl}$. 
Martin and Steenbock construct a complex of groups $G(\mc{K})$ over $\mc{K}$ where the local groups are trivial except for:
\begin{itemize}
	\item The local group at $u_A$ is $A$,
	\item the local group at $u_B$ is $B$,
	\item and the local group at the apex vertex is $\Z/d\Z$ (recall $d$ is defined so that the relator is $w^d$).
\end{itemize}
All local maps are trivial. We define the twisting elements according to a more complicated scheme: first, Martin and Steenbock assign a group element $h_f$ to each arrow $f$ of the scwol $\mc{K}$. To write down these elements, we use some additional notation: if the small cancellation relator we are using is $w^d$ where $w$ is written as: 
\begin{equation}
	w=a_0b_0a_1b_1\ldots,a_nb_n \qquad a_i\in A,\,b_i\in B, \label{relator spelling}\end{equation} then let $w_i$ be the first $i$ letters of $w$ as written in \eqref{relator spelling}. If $\sigma,\sigma'$ are distinct simplices of $G(\mc{K})$ (which correspond to vertices of $\mc{K}$), and $\sigma\subseteq \sigma'$, then $(\sigma,\sigma')$ will be used to denote the arrow in $\mc{K}$ from $\sigma'$ to $\sigma$. Also recall the notations $v_i,e_i,\sigma_i$ from \Cref{keeping vertices and edges straight}. 
We assign the $h_f$ as follows:
\begin{itemize}
	\item for each $1\le i \le 2N$, $h_{(v_{i-1},e_{i-1})}= w_{i-1}\inv$,
	\item for each $1\le i \le 2N$, $h_{(v_{i-1},\sigma_i)} = w_{i-1}\inv$,
	\item for each $1\le i\le 2N$, $h_{(L, \sigma_i)} = w_i\inv$, 
	\item for each $1\le i\le 2N$, $h_{(v_i,\sigma_i)} = w_i\inv$,
	\item $h_{e_{2N},\sigma_{2N}} = w_{2N}\inv$,
	\item and $h_f$ is trivial for all other arrows $f$.
\end{itemize}
For any $e,f$ a pair of composable arrows in $\mc{K}$, the twisting element $g_{f,e}$ is defined as $g_{f,e} = h_{f}h_{e} h_{fe}\inv$.

Much like a graph of groups, the complex of groups structure of $G(\mc{K})$ carries both algebraic and geometric information about the group $G$.

\begin{theorem}[{\cite[Propositions 2.5, 2.8, 2.10]{MartinSteenbock}}]
	\label{UCover of cmplx X}
	There exists a simply connected $C'(\frac16)$ polygonal complex $X$ with an action of $G$ so that the quotient of the action $\leftQ{X}{G}$ is isomorphic as a complex of groups to $G(\mc{K})$. 
	Also, the fundamental group (in the category of complexes of groups) of $G(\mc{K})$ is isomorphic to $G$. 
\end{theorem}
In other words, $G(\mc{K})$ is a developable complex of groups with $X$ as its universal cover (as a complex of groups).
Since $G(\mc{K})$ is a finite complex of groups, the action of $G$ on $X$ is cocompact. 
We will refer to $X$ as the \textbf{developed complex} associated to the complex of groups structure $G(\mc{K})$.
Our main concern is that $X$ is polygonal and $G$ admits a cocompact action on $X$, so we will not discuss the fundamental group of a complex of groups further. 
The interested reader should refer to \cite[Section III.$\mc{C}$.3]{BridsonHaefliger}.
\begin{remark}
	\label{Rmk:MultipleFactors}
	The construction described in this subsection directly generalizes to all $C'(\frac16)$--small cancellation free products with $n \geq 2$ factors and finitely many relations with the following adjustments.
	The complex $L'$ will be a star where the local groups of valence 1 vertices are each of the factor groups and all other local groups are trivial.  
	Distinct relations will give distinct labelled polygons $R_{0, \simpl}$ whose boundaries are glued to $L'$ to build a complex $K_\simpl$ that gives rise to a complex of groups over the scwol over $K_\simpl$ as above.
\end{remark}

\subsection{The relatively hyperbolic structure for $G$ with respect to the vertex groups}
\label{S: relhyp}

A priori quotients of free products need not inherit relative hyperbolicity from the vertex groups.  For example, the direct product of two finitely generated infinite groups $A \times B$ arises as a quotient of the free product $A * B$ by finitely many words.  It is thus, essential to our arguments that we restrict our attention to quotients satisfying the $C'(\frac{1}{6})$ small cancellation in \Cref{SS:smallCancel}.   


An important aspect of working with the complex $X$ is that its large-scale geometry is closely related to the relative hyperbolicity of $G$.   The following result is originally due to Pankrat'ev \cite{Pankratev} and also proved by Steenbock \cite[Theorem~1.1]{SteenbockRS}:
\begin{proposition}\label{P: hyp rel vertex}
	Let $\mc{Q} = \{A,B\}$. Then $(G,\mc{Q})$ is a relatively hyperbolic pair.
\end{proposition}

\begin{proof}
	Consider a finite relative presentation  $\cyc{A,B \, | \, R}$ for $G$ relative to $\mc{Q}$ (see \cite[Section 3.5]{Hruska2010} for details about finite relative presentations). 
	Let $w$ be a reduced word over the alphabet $A\cup B$ so that $w=_G 1_G$.
	Since $R$ satisfies the $C'(\frac16)$ condition, \cite[Chapter V, Theorem 9.3]{LyndonSchupp} implies that there exists a reduced factorization $w=usv$ and a reduced word $r$ over the alphabet $A\cup B$ that is a cyclic conjugate of some $r'\in R$, so that $r =st$ and $|s|>\frac12 |r|$. 
	By applying Dehn's algorithm (see e.g. \cite[III.$\Gamma$.2.4]{BridsonHaefliger}), we conclude that:
	\[w = r_1'r_2'\ldots r_k'\]
	where each $r_i'$ is a cyclic conjugate of some $r_i\in R$ and $k\le |w|$. 
	Hence there exists a linear relative isoperimetric function for the finite relative presentation $\cyc{A,B \, | \, R}$, so $(G,\mc{Q})$ satisfies \cite[Definition 3.7 (RH-6)]{Hruska2010}, one of the equivalent definitions given in \cite{Hruska2010} for a relatively hyperbolic pair. 
\end{proof}

A result of Charney and Crisp \cite[Theorem 5.1]{CharneyCrisp} asserts that any space admitting a cocompact action by a relatively hyperbolic group, in which the peripherals are maximal among subgroups that act with nontrivial fixed point, is equivariantly quasi-isometric to any coned--off Cayley graph.  
Recall that a \textbf{coned-off Cayley graph} for $(G, \mc{P})$ is obtained from any finite generated set $S$ of $G$ by attaching a cone over distinct cosets of each $P \in \mc{P}$ in $\Cay(G, S)$.
The cocompactness of the action on $G$  implies the following key geometric fact:

\begin{proposition}\label{CCrestatement}
	Let $\Gamma(G,\mc{Q})$ be a coned-off Cayley graph for $G$ with respect to some finite generating set. 
	The action of $G$ on any hyperbolic graph $\Gamma$ where
	\begin{enumerate}
		\item $G$ acts cocompactly by isometries,
		\item the infinite vertex stabilizers are maximal parabolics,
		\item the edge stabilizers are finite, and
		\item every maximal parabolic stabilizes a vertex
	\end{enumerate}
	induces a $G$-equivariant quasi-isometry $\Gamma(G,\mc{Q})\to \Gamma$. 
	
	In particular, $X^{(1)}$ is a Gromov hyperbolic graph.
\end{proposition}

Another crucial property of small cancellation free products of non-elementary relatively hyperbolic groups is that they posses several relatively hyperbolic structures.  Most importantly, these different structures are naturally witnessed by actions of the group on certain spaces associated to the decompositions as complexes of groups described in \Cref{S: complex} and \Cref{SS:blowupConstruction}.
An easy consequence of \Cref{P: hyp rel vertex} is the following:

\begin{proposition}\label{rh structure P}
	If $(A,\mc{P}_A)$ are $(B,\mc{P}_B)$ are relatively hyperbolic pairs and $\mc{P} = \mc{P}_A \cup \mc{P}_B$, then $(G,\mc{P})$ is a relatively hyperbolic pair.
\end{proposition}

\begin{notation}
	\label{notation:multiple_factors}
	For the remainder of this paper, we will set $\mc{Q} =\{A,B\}$ so that $(G,\mc{Q})$ is the relatively hyperbolic structure of $G$ from \Cref{P: hyp rel vertex} where the peripheral subgroups are entire vertex groups with the understanding that everything generalizes to any finite collection of factor groups as in \Cref{Rmk:MultipleFactors}.
	
	We will use $(G,\mc{P})$ to denote the relatively hyperbolic structure with $\mc{P} = \mc{P}_A \cup \mc{P}_B$  coming from the peripherals of the vertex groups as in \Cref{rh structure P}. 
	Likewise, we will use $\partial_{\mc{Q}}G$ and $\partial_{\mc{P}}G$ to denote the respective Bowditch boundaries of $(G,\mc{Q})$ and $(G,\mc{P})$. 
\end{notation}

%

\subsection{Acylindricity of relative hyperbolic group actions}
\label{S: acylindricity}

Acylindrical actions on general hyperbolic spaces were defined by Bowditch \cite{BowditchTightGeodesics} extending prior notions due to Sela \cite{Sela97} and Delzant \cite{DelzantAcyl}.  
The class of \emph{acylindrically hyperbolic groups} is quite broad and natural.  
For example, relatively hyperbolic groups provide a rich source of acylindrically hyperbolic groups  (see \cite{OsinAH} and references therin for more on acylindrical hyperbolicity).  
As will be apparent in \Cref{P: hypergraph separation}, our methods only require controlling orbits of certain quasiconvex subsets of complex formed by subdividing $X$.  
We will make use of the following formulation of acylindrical actions introduced by Abbott and Manning \cite[Definition~3.1]{AbbottManning}.

\begin{definition}
	\label{def:AQI-cobddSubset}
	Let $K$ act by isometries on a hyperbolic metric space $Y$, and let $H$ be a subgroup that is quasi-isometrically embedded by the action.  The action of $K$ on $Y$ is \textbf{acylindrical along} $H$ if for some (equivalently any) $H$--cobounded subspace $W \subset Y$ the following condition holds: 
	
	For any $\varepsilon \geq 0$, there exists constants $D = D(\varepsilon, W)$	and $N = N(\varepsilon, W)$ such that any collection of distinct cosets $ \{ g_1 H, \dotsc, g_k H \} $ such that 
	\[
	\# 		\left\{ 	\,		n 		\, 		: 		\, 		\diam \left( \bigcap_{i = 1}^k {\mc N}_{\varepsilon} (g_i W) \right) > D  		\,	\right\}		\leq N.
	\]
	Groups $K$ that admit isometric actions on non-elementary hyperbolic spaces that are acylindrical along the entire group are called \textbf{acylindrically hyperbolic}.
\end{definition}

In \cite[Theorem 3.2]{AbbottManning}, Abbott and Manning  show that the notion of acylindricity in \Cref{def:AQI-cobddSubset} agrees with the notion defined by Bowditch.  

In the setting of 
$C'(\frac16)$--small cancellation free products, acylindricity of the $G$ action on $X$ will be established in       
\Cref{P: acylindrical}.



\subsection{Small cancellation blow-up}
\label{SS:blowupConstruction}

The geometry of the factor groups is not reflected in the developed complex $X$. 
We will see how to enlarge $X$ to a new complex $\EG$ that captures the cubical geometry of the factor groups using a construction of Martin and Steenbock \cite[Definition~2.12]{MartinSteenbock} (see also prior work of Martin \cite[Definition~2.2]{Martin:NPC_boundary}).

Let $EA$ and $EB$ denote $\CAT(0)$ cube complexes on which $A$ and $B$ respectively act cocompactly.  
Originally, Martin and Steenbock require that the actions on $EA$ and on $EB$ are also proper in order to conclude that the resulting blown-up space $\EG$ admits a geometric action by $G$.
Note that the proof of \cite[Proposition~2.14]{MartinSteenbock} only uses properness of these actions to show that the action of $G$ on $\EG$ is again proper.  
In particular, when $A$ and $B$ only admit relatively geometric actions on $EA$ and $EB$ respectively, the action of $G$ on $\EG$ is still cocompact, which we record below in \Cref{EG cocompactness}. 

\begin{figure}
\centering 
%
%
%
%
%
%
%
%
%

\begin{tikzpicture}[
	scale=0.1,
	ultra thick
	]
	
	\def\rad{1}
	
	\draw[blue] (17,3) -- (13, 9); 
	\draw[red] (9, 13) -- (3, 17);
	\draw[blue] (-3, 17) -- (-33, 13) ;
	\draw[red] (-37, 9) -- (-41, 3); 
	\draw[blue] (-41, -3) -- (-39, -9); 
	\draw[red] (-35, -15) -- (-3, -17); 
	\draw[blue] (3, -17) -- (9, -13);
	\draw[red] (13, -9) -- (17, -3);
	
	\draw (13, 9) -- (9, 13);
	\draw (3, 17) -- (-3, 17);
	\draw (-33, 13) -- (-37, 9);
	\draw (-41, 3) -- (-41, -3);
	\draw (-39, -9) -- (-35, -15);
	\draw (-3, -17) -- (3, -17);
	\draw (9, -13) -- (13, -9);
	\draw (17, -3) -- (17,3);
	
	\fill[blue] (17,3) circle (\rad); 
	\fill[blue] (13, 9) circle (\rad);
	\fill[red] (9, 13) circle (\rad);
	\fill[red] (3, 17) circle (\rad);
	\fill[blue] (-3, 17) circle (\rad);
	\foreach \t/\T in {0.66/0.34, 0.34/0.66}{
		\fill[blue] (-3*\t-33*\T, 17*\t + 13*\T ) circle (\rad*0.75);
	}
	\fill[blue] (-33, 13) circle (\rad); 
	\fill[red] (-37, 9) circle (\rad);
	\fill[red] (-41, 3) circle (\rad);
	\fill[blue] (-41, -3) circle (\rad);
	\fill[blue] (-39, -9) circle (\rad);
	\fill[red] (-35, -15) circle (\rad);
	\foreach \t/\T in {0.66/0.34, 0.34/0.66}{
		\fill[red] (-35*\t-3*\T, -15*\t - 17*\T ) circle (\rad*0.75);
	}
	\fill[red] (-3, -17) circle (\rad);
	\fill[blue] (3, -17) circle (\rad);
	\fill[blue] (9, -13) circle (\rad);
	\fill[red] (13, -9) circle (\rad);
	\fill[red] (17, -3) circle (\rad);

	\def\xoffset{50}
	
	\draw[->] (25, 0) --(\xoffset - 15, 0);

	\draw (7+\xoffset,3) -- (3+\xoffset, 7) -- (-3+\xoffset, 7) -- (-7+\xoffset, 3) -- (-7 +\xoffset, -3) -- (-3 + \xoffset, -7) -- (3+\xoffset, -7) -- (7+\xoffset, -3) -- cycle;
	
	\foreach \x/\y in {7/3, 3/7, -3/7, -7/3, -7/-3, -3/-7, 3/-7, 7/-3}{
		\fill[blue] (\x+\xoffset, \y) circle (\rad);
	}
	
	\foreach \x/\y in {3/7, -7/3, -3/-7, 7/-3}{
		\fill[red] (\x+\xoffset, \y) circle (\rad);
	}
	
	%
	%
	%
	
	
	
	
\end{tikzpicture}
\caption{A blown-up polygon in $\EG$ and its image in the polygonal complex $X$}
\label{F:polygon-blowup}
\end{figure}

We summarize key aspects of this construction to establish notation and  for completeness.
The reader should think of $\EG$ as replacing vertices of $X$ with copies of either $EA$ or $EB$ in accordance with the local group.  
In doing this, verticies of polygons are blown up to geodesic segments within the vertex space as in \Cref{F:polygon-blowup}.  
For precise details about the attaching maps that ensure $G$ still acts isometrically on this new space see \cite[Definitions~2.16-2.18]{MartinSteenbock}.  
For our purposes, it is enough to know the existence of a projection $p: \EG \to X$.
The geometry of $\EG$ is largely determined by the fibers over the vertices of $X$. These fibers are:
\[
EG_x := p^{-1}(x) \cong  
\begin{cases}
EA	&	\text{when } G_x = A^g \text{ for some $g\in G$}	\\
EB	&	\text{when } G_x = B^g	\text{ for some $g\in G$}\\
\text{a single point}	&	\text{otherwise.}
\end{cases}
\]
The interior of each edge and polygon of $X$ lifts uniquely to an edge or polygon of $\EG$. 
We will refer to $EG_x$ as a \textbf{fiber complex} and $G_x$, its stabilizer, as a \textbf{fiber group}. 	
We call the complex $\EG$ the \textbf{blow-up} of the developed complex $X$ with respect to the actions of $A$ and $B$ on $EA$ and $EB$.

Here is a very simple example of this construction:
\begin{example}\label{simple example}
Let $A = \Z,$ let $B = \Z$ and let $G= A*B$. Let $EA$ and $EB$ be the standard Cayley graphs for the presentations $\Z\cong\cyc{a_0}$ and $\Z\cong \cyc{b_0}$ respectively. 
The complex of groups in this case is the standard graph of groups presentation for $G=F_2$ and its development is the standard Bass-Serre tree with distinguished vertices $v_A$ and $v_B$ stabilized by $A$ and $B$ respectively. Note that $G$ does not act properly on the Bass-Serre tree $T$. 
In this case, the blow-up $\EG$ will be a $3-$valent tree on which $G$ acts properly. 

The space $\EG$ is constructed is as follows: pick distinguished vertices $x_A$ and $x_B$ in $EA$ and $EB$. For each vertex $v$ of $T$, let $\EG_v$ be a copy of either $EA$ or $EB$ depending on whether $v$ is in the orbit of the vertex of $T$ stabilized by $A$ or in the orbit of the vertex stabilized by $B$. Let $x_v$ be the copy of $x_A$ or $x_B$ as appropriate. 
We now glue together these spaces as follows: for each $a\in A$, there is an edge in $T$ from $v_A$ to $a \cdot v_B$ we connect the vertex $a\cdot x_A\in EA$ to the vertex $x_{a\cdot v_B}\in EG_{a\cdot v_B}$ by an edge. 
Likewise, for each $b\in B$, connect the vertex $b\cdot x_B\in EB$ to $x_{b\cdot x_A}\in EG_{b\cdot x_A}$ (except there should only be one edge between $x_A$ and $x_B$). 
We can then extend these choices equivariantly to obtain a $3-$valent tree. 
\end{example}

\begin{figure}[h]
\centering
\includegraphics[trim=10 80 310 80, clip, scale = 0.4]{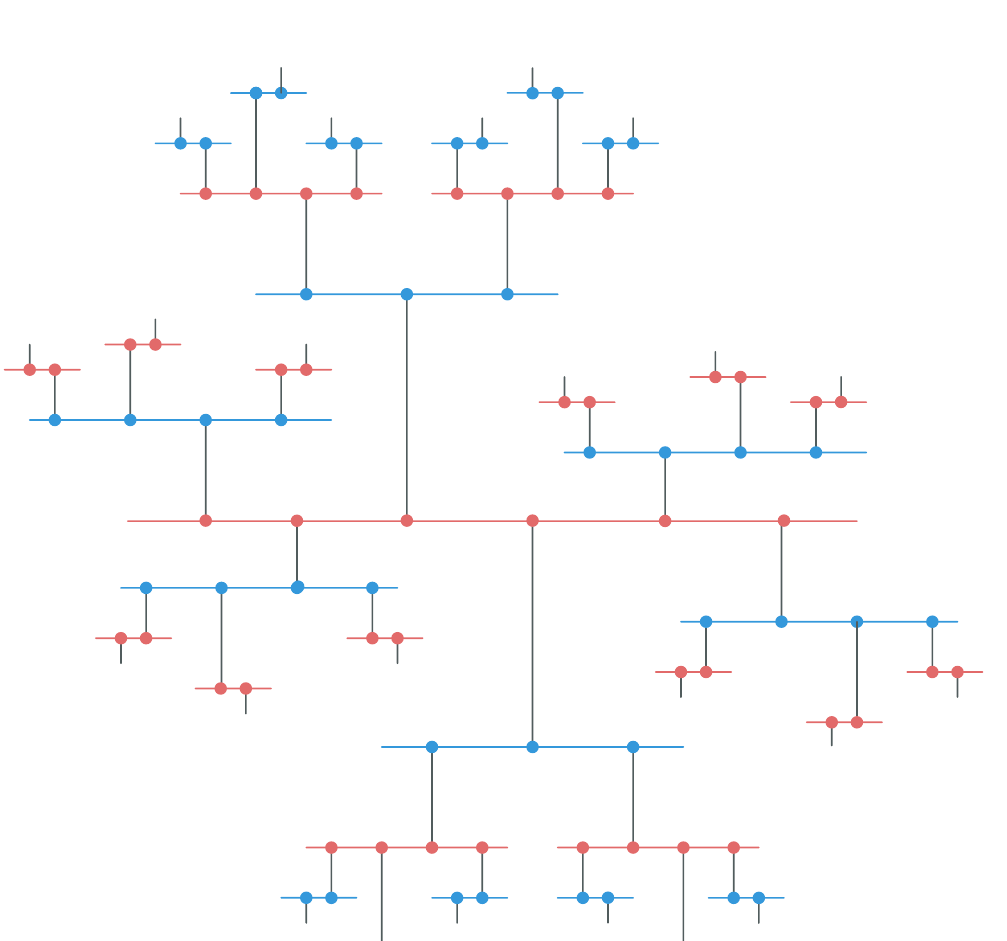}
\qquad \raisebox{2.75cm}{$\xrightarrow{\hspace{2cm}}$} \qquad
\includegraphics[trim=520 80 60 80, clip, scale = 0.4]{tree_of_spaces-v2.pdf}
\caption{The blow-up described in \Cref{simple example} is the standard tree of spaces over the Bass-Serre tree}
\label{fig:enter-label}
\end{figure}

In general, we will only have a relatively geometric action on the fiber complexes, so the resulting action on $\EG$ will not be proper.

Using parts of \cite[Theorem 2.4]{Martin:NPC_boundary} that do not require properness, we obtain the following:
\begin{proposition}\label{EG cocompactness}
The action of $G$ on $\EG$ is cocompact, and $\EG$ is simply connected.
\end{proposition}
Note that the properness hypotheses of \cite[Theorem 2.4]{Martin:NPC_boundary} are not used in the cocompactness and contractibility arguments.

\section{The Construction of the Walls}
\label{S: wall construction}

In this section, we review some of the objects Martin and Steenbock construct to help build their wallspace structure on $\EG_{bal}$, a subdivision of $\EG$ that is defined in  \Cref{WallsFromX}. 

We formalize \textbf{hyperstructures}, which are used implicitly in \cite{MartinSteenbock} (see \Cref{D:hyperstructure}). Hyperstructures are inspired by hyperplanes in cube complexes. Similar ideas have been used to construct wallspaces for cubulation, see for example \cite{WiseSmallCancellation}. Like hyperplanes, they have hypercarriers and dual edges. We will also encounter \textbf{hypergraphs} in $X_{bal},$ a subdivision of $X$ that makes the projection $\EG_{bal}\to X_{bal}$ a combinatorial map. Hypergraphs are either hyperstructures of $X_{bal}$ or projections of hyperstructures of $\EG_{bal}$ to $X_{bal}$. 
The notation in this section is based on \cite{MartinSteenbock} to help the reader refer between our work and theirs. Superscripts are intended to indicate which type of hyperstructure an object arises from, and subscripts indicate a dual edge that determines the hyperstructure an object arises from. 

\subsection{$\Omega$--hyperstructures}	\label{S: hyperstructures}

Let $\Omega$ be a complex whose cells are polygons with an even number of edges. We can place an equivalence relation $\sim_{opp}$ on the edges of $\Omega$ as follows: for all edges $e$, $e\sim e$ and if $U$ is a polygon in $\Omega$ and $e_1,e_2$ are diametrically opposite edges of the polygon $U$, then $e_1\sim_{opp} e_2$. 
Then $\sim_{opp}$ can be extended to an equivalence relation on $\Omega$ by taking the transitive closure. 

\begin{example}
If $\Omega$ is a square complex, $e_1\sim_{opp} e_2$ if and only if they are dual to the same hyperplane. 
\end{example}

More generally, if $\Omega$ is a complex where every cell is either a polygon with an even number of edges or an $n$-cube, we can place a similar relation on $\Omega$: let $U$ be a polygon or $n$-cube of $\Omega$ where $n\ge 2$. If $e_1,e_2$ are edges in $U$, we say $e_1\sim_{opp} e_2$ when:
\begin{enumerate}
\item if $e_1=e_2$, then $e_1\sim_{opp}e_2$
\item if $U$ is a polygon, then $e_1\sim_{opp} e_2$ when they are diametrically opposed in $P$, and
\item if $U$ is an $n\ge 2$ cube, then $e_1\sim_{opp} e_2$ when they are dual to the same midcube. 
\end{enumerate} 
As before, $\sim_{opp}$ extends to an equivalence relation by taking the transitive closure. 

\begin{definition}
\label{D:hyperstructure}
Let $e$ be an edge in $\Omega$. The \textbf{$\Omega-$hyperstructure} associated to $e$, $W_e^\Omega$, is the subspace of $\Omega$ constructed as follows: for each pair $e_1,e_2\in [e]_{\sim_{opp}}$ where $e_1,e_2$ are edges of an $(n \geq 2)$--cell $U$, 
\begin{enumerate}
	\item when $U$ is a polygon, add straight segments joining the center of $U$ to the midpoints of $e_1$ and $e_2$ or
	\item when $e_1$ and $e_2$ are dual to a midcube of an $n$-cube $U$ in which case, include the midcube of $U$ that is dual to $e_1$ and $e_2$ in $W_e^\Omega$.
\end{enumerate} 
\end{definition}
The \textbf{carrier} of the hyperstructure $W_e^\Omega$ denoted $Y_e^\Omega$ is the union of all cells whose interior intersects $W_e^\Omega$.  

For example, if $\Omega$ is any cube complex, the $\Omega-$hyperstructure associated to an edge $e$ is the hyperplane dual to $e$. 

The \textbf{abstract carrier of $W_e^\Omega$}, is the complex $\tilde{Y}_e^\Omega$ constructed as follows:
\begin{enumerate}
\item For each pair of diametrically opposed edges of a polygon $U$ in $[e]_{\sim_{opp}}$ include a copy of $U$ with the edges $e_1,e_2$ labeled and the remaining edges unlabeled.
\item For each collection of edges in $[e]_{\sim_{opp}}$ dual to the same midcube $M$ of a maximal cube $C$, include a copy of $C$ with copies of the edges dual to $M$ labeled and the remaining edges unlabeled.

\item If $F,F'$ are cells of $Y_{e}^\Omega$ with labeled edges $e,e'$ respectively so that $e,e'$ have the same image $e_{F,F'}\in [e]_{opp}$, let $E,E'$ be the maximal subcomplexes of $F,F'$ respectively whose images in $\Omega$ agree and contain $e_{F,F'}$. 
Then glue $F,F'$ in $Y_e^\Omega$ by identifying points in $E,E'$ that have the same image in $\Omega$. \label{D:OverlapCarrierPolys}
\end{enumerate}

We remark that when polygons in $\Omega$ have connected intersections (for example as a consequence of a small cancellation condition like \Cref{D: smcancel fp}) then in \Cref{D:OverlapCarrierPolys} the paths $\rho$ and $\rho'$ can be chosen to be maximal paths whose images necessarily contain an edge $e_{U, U'} \in [e]_{\sim_{opp}}$. 
Note also that there is a natural surjective  immersion $\tilde{Y}_e^\Omega\to Y_e^\Omega$ from the abstract carrier to the carrier. 

\begin{example}
In \Cref{F: abstractcarrier}, we show a square complex with one of its hyperplanes and its abstract carrier. In this case, the abstract carrier immerses but does not embed. 
\begin{figure}
	\includegraphics[scale=.75]{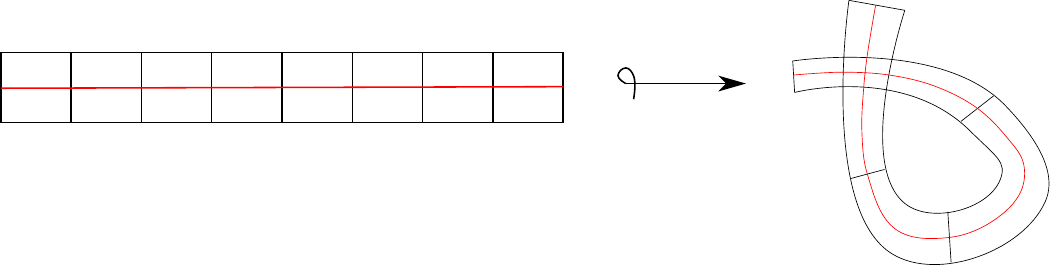}
	\caption{On the right is a square complex with a hyperplane drawn in red. On the left is its abstract carrier.}\label{F: abstractcarrier}
\end{figure}
\end{example}

The following lemma will be prove useful later:
\begin{lemma}\label{L: hyperstructure stab cocompact}
Let $\Omega$ be a complex where every cell is a polygon with an even number of cells or an $n$--cube for some $n$. Suppose that $K$ acts combinatorially and cocompactly on $\Omega$. If $W_e^\Omega$ is an $\Omega$--hyperstructure, then $\Stab_K(W_e^\Omega)$ acts cocompactly on $W_e^\Omega$. 
\end{lemma}

\begin{proof}
By construction hyperstructures are closed.  
Subsets of the quotient are closed if and only if their preimage is closed.  
Thus, the image of $W_e^\Omega$ is compact.  
Cocompactness follows because the the restriction of the quotient map to $W_e^\Omega$ is induced by the action of $\Stab_K(W_e^\Omega)$.
\end{proof}

\subsection{Hypergraphs coming from the subdivided complex $X_{bal}$}
\label{WallsFromX}

We need to subdivide $X$ and $\EG$ to obtain the appropriate walls.  Specifically, we need to ensure that hyperstructures cannot enter and exit a single polygon too close together.

\begin{definition}[{\cite[Definition 3.3]{MartinSteenbock}}]
\label{far apart}
Let $U$ be a polygon of a $C'(\frac16)$-polygonal complex $X$ and $\tau_1,\tau_2$ be two simplices of the boundary $\partial U$. We say that $\tau_1,\tau_2$ are \textbf{far apart} in $U$ if no path $\alpha$ in $\partial U$ containing both $\tau_1 $ and $\tau_2$ is a concatenation of strictly less than four pieces.
\end{definition}

Note that the notion of piece in \Cref{far apart} is \cite[Definition~2.9]{MartinSteenbock}, which defines a \emph{piece} of a polygonal complex as an injective path $\rho$ contained in the boundary of two polygons, whose boundaries cannot be identified homeomorphically while keeping  $\rho$ fixed pointwise.  

Recall from the construction of $\EG$ that the projection $p: \EG \to X$ is injective on edges whose interior is disjoint from every fiber spaces.  These edges are referred to as \textbf{horizontal edges} of the blow-up $\EG$.  Edges whose interiors are contained in a fiber space are called \textbf{vertical edges} (see also \cite[Definition~2.17]{MartinSteenbock})

We subdivide $\EG$ in two steps with respect to an integer $k \geq 0$.  
First, subdivide every vertical edge once and then cubically subdivide the fiber complexes to match the vertical subdivisions.  This guarantees that polygons are glued to fiber complexes along paths of even length.  
Second, subdivide each horizontal edge into $k$ edges.  
The projection $p: \EG\to X$ will remain a combinatorial projection when every edge of $X$ is subdivided into $k$ edges. 

The inequality in \Cref{D: smcancel fp} is strict, so for all $k$ sufficiently large the resulting subdivision $\EG_{bal}$ of $\EG$ and subdivision $X_{bal}$ of $X$ have the following two properties:
\begin{enumerate}
\item If $U$ is a polygon of $\EG_{bal}$ or $X_{bal}$, then $U$ has an even number of sides,
\item Under the natural projection $p:\EG_{bal}\to X_{bal}$, every pair of diametrically opposed edges of a polygon $U$ in $\EG_{bal}$ projects to a pair of edges that are far apart in $p(U)$ in $X_{bal}$ \cite[Lemma 3.24]{MartinSteenbock}. 
\end{enumerate}
The subdivisions $\EG_{bal}$ and $X_{bal}$ are called \textbf{balanced subdivisions}.

For each edge $e$ of $X_{bal}$ we define $Z_e^X$ to be the $X_{bal}$-hyperstructure associated to $e$ which we refer to as the \textbf{hypergraph of $X_{bal}$ associated to $e$}. The \textbf{carrier} or \textbf{hypercarrier} $Y_e^X$ of $Z_e^X$ is the smallest subcomplex of $X_{bal}$ that contains $Z_e^X$. 

Rather than constructing a unified wallspace structure in $\EG_{bal}$ like Martin and Steenbock do, we will construct a sufficient family of codimension-$1$ subgroups that satisfy the hypotheses of \Cref{t:Rel BW} from the hyperstructures of \emph{both} $X_{bal}$ and $\EG_{bal}$. 

\subsection{The hyperstructures of $\EG_{bal}$}
\label{WallsFromEG}

In \cite{GenFine}, we prove that any pair of distinct points in the Bowditch Boundary of a relatively geometrically cubulated group are separated by the limit set of a hyperplane stabilizer. 
An intuitive way to try to separate points in the Bowditch Boundary of a single fiber group is to use the hyperplanes of the fiber complexes, but we need to extend the hyperplanes to ensure that they separate $\EG$. 

If $e$ is an edge of a fiber complex of $\EG_{bal}$, we define ${W}_e^{\EG}$ to be the $\EG_{bal}-$hyperstructure associated to $e$, which is two-sided and contractible by \cite[Lemma 3.34]{MartinSteenbock}. 

In the specific case of $(G,\mc{P})$, the codimension--$1$ subgroups that we use to separate points in $\partial_{\mc{P}}G$ that lie in the Bowditch Boundary of a single fiber group will be finite index subgroups of the $\Stab_G (W_e^{\EG})$. 

The far apart condition will ensure that projections of $W_e^{\EG}$ have properties similar to the $X_{bal}-$hyperstructures. 
We now introduce a few facts about these \textbf{projected hypergraphs}.

\begin{definition}
Let ${W}_e^{\EG}$ be an $\EG_{bal}-$hyperstructure. The \textbf{projected hypergraph $Z_e^{\EG}$} is the image of ${W}_e^{\EG}$ under the projection $p:\EG_{bal} \to X_{bal}$. 
\end{definition}

Unlike the hypergraphs $Z_e^X$ associated to classes of opposite edges in $X_{bal}$, it is possible that the intersection of a projected hypergraph $Z_e^{\EG}$ with the boundary of a polygon in $X_{bal}$ includes a vertex. 
\begin{definition}
The \textbf{hypercarrier $V_e^{\EG}$ associated to $Z_e^{\EG}$} is the union (in $X_{bal}$) of the polygons of $X_{bal}$ whose interior intersects $Z_e^{\EG}$.
\end{definition}

We can also construct an \textbf{abstract hypercarrier $\tilde{V}_e^{\EG}$ associated to $Z_e^{\EG}$} as follows: 
let $\tilde{Y}_e^{\EG}$ be the abstract carrier obtained by viewing ${W}_e^{\EG}$ as an $\EG_{bal}-$hyperstructure. 
For each polygon $\hat{U}$ of $\hat{Y}_e^{\EG}$, let $q(\hat{U})$ be the image in $X_{bal}$ after mapping $\hat{U}$ into $\EG_{bal}$ and applying the projection $p:\EG_{bal} \to X_{bal}$. Note that $q(\hat{U})$ is a polygon of $X_{bal}$ (but may have fewer boundary edges than $\hat{U}$). 
For each $\hat{U}$, $\hat{V}_e^{\EG}$ contains a copy of $q(\hat{U})$.
If $\rho$ is common piece of $\hat{U_1}$ and $\hat{U_2}$, the copies of $q(\hat{U_1})$ and $q(\hat{U_2})$ are glued together along the image of $q(\rho)$ which may be either a path or a vertex. 

Martin and Steenbock's argument showing that the natural immersion of a hypercarrier of $W_e^X$ is an embedding also applies to the natural immersion $\tilde{V}_e^{\EG}\to V_{e}^{\EG}$:
\begin{proposition}[{\cite[Theorem 3.10]{MartinSteenbock}}]\label{projected hypercarrier embed}
The natural immersion $\tilde{V}_e^{\EG}\to V_e^{\EG}$ is an embedding.
\end{proposition}

The equivariance of the projection $\EG_{bal}\to X_{bal}$ also implies that $\Stab_G({W}_e^{\EG}) \le \Stab_G(Z_e^{\EG})$.

In the following, we adopt the convention that sets $Z = Z_e^X$ and $Z = Z_e^{\EG}$ are subsets of $X_{bal}$ and are called \textbf{hypergraphs (of $X_{bal}$)}.  

\begin{example}
We briefly recall \Cref{simple example} to give very basic examples of these objects. For an edge $e$ in $EG_v$, ${W}_e^{\EG}$ is the midpoint of $e$. Its carrier is $e$.  Likewise for an edge $f$ in the Bass-Serre tree, $Z_f^{X}$ is the midpoint of the edge $f$. The carrier of $Z_f^X$ is $f$.
\end{example}

\section{Realizing and working with the Bowditch Boundary}\label{S: Peripheral Structures}

In this section, we build tools that will help us turn stabilizers of hypergraphs in $X_{bal}$ and $\EG-$hyperstructures into codimension--$1$ subgroups that separate points in the Bowditch boundary. Our ultimate goal is to use these subgroups for \Cref{t:Rel BW}. A given group $K$ may admit multiple relatively hyperbolic structures, which yield different Bowditch boundaries.  In \Cref{S: separation refined}, we prove that we can often separate points in the Bowditch boundary of a refined peripheral structure by separating their preimages in the boundary of the original structure. The remainder of the section, especially \Cref{T: separation criterion}, adapts prior work of the authors with Groves \cite{GenFine} on separating pairs of points in Bowditch boundaries with quasiconvex subgroups.


\subsection{Separating Points in the Bowditch Boundary of a Refined Peripheral Structure}\label{S: separation refined}

Here we consider $(K,\mc{R})$, a 
refined peripheral structure for a relatively hyperbolic pair $(K,\mc{D})$ where each $R\in\mc{R}$ is conjugate into some $D\in\mc{D}$. In our application, $(K,\mc{D})$ will be derived from a blowup construction modeled after $\EG_{bal}$ that we introduce in \Cref{S: blowup main}. In this subsection, we lay the groundwork so that in \Cref{S: blowup main}, we can apply \Cref{t:Rel BW}, the relatively geometric cubulation criterion, if we can show that:
\begin{enumerate}
\item any two points in the Bowditch Boundary of $(K,\mc{D})$ can be separated by the limit set of a full relatively quasiconvex (in $\mc{D}$) codimension--$1$ subgroup, and
\item any two points in the Bowditch Boundary of $(K,\mc{R})$ that lie in the limit set of $D^g$ can be separated by a full relatively quasiconvex codimension--$1$ subgroup. 
\end{enumerate}


The coarser relative structure $(K,\mc{D})$ is a sometimes also called a \emph{parabolically extended structure} for $(K,\mc{R})$ \cite{WenyuanYang2014}. Yang proves the following useful relationship between their Bowditch boundaries:
\begin{proposition}[{ \cite[Lemmas~4.13-4.14]{WenyuanYang2014}}]
\label{P: peripheral boundary map}
There exists a continuous $K$--equivariant surjective map $\phi: \partial_{\mc{R}} K \to \partial_{\mc{D}}K$. Further, for any conical limit point $x$ of $\partial_{\mc{D}}K$, $\phi\inv(x)$ consists of a single point.
\end{proposition}
In particular, the map $\phi$ collapses the entire limit set $\Lambda_{\mc{R}} D^g$ of a maximal parabolic in $(G,\mc{R})$ to the parabolic point it stabilizes in $\partial_{\mc{D}}K$.  
Separating points in the Bowditch Boundary is sometimes more straightforward in one relative structure than another.  

\begin{proposition}\label{P: projection preimages separation}
Let $H\le G$ be a relatively quasiconvex in $(K,\mc{D})$ and let $\Lambda_{\mc{D}}H$ and $\Lambda_{\mc{R}}H$ be the limit sets of $\Lambda$ in $(K,\mc{D})$ and $(K,\mc{R})$ respectively.

If $x$ and $y$ lie in distinct $H$--components of $\partial_{\mc{D}}K\setminus \Lambda_{\mc{D}} H$, then any $x_0\in\phi\inv (x)$ and $y_0\in\phi\inv (y)$ lie in distinct $H$--components of $\partial_{\mc{R}} K\setminus \Lambda_{\mc{R}} H$. 
\end{proposition}

\begin{proof}
Since $\phi$ is continuous and equivariant, if $x_0$ and $y_0$ were in the same $H$--component of $\partial_{\mc{D}'} K\setminus \Lambda_{\mc{D}'} H$, their images under $\phi$ must lie in the same $H$--component of $\partial_{\mc{D}} K\setminus \Lambda_{\mc{D}} H$.
\end{proof}

\subsection{Fine hyperbolic graphs, generalized fine actions and the Bowditch Boundary}
\label{S: fine genfine}

Later, via a more general framework, we will see that $X_{bal}\oskel$ is a fine hyperbolic graph that can be used to realize the Bowditch Boundary of $(G,\mc{Q})$. Similarly, $\EG_{bal}\oskel$ can be used to realize the Bowditch Boundary of $(G,\mc{P})$. In this subsection, we introduce some technical details about the relationship between fine hyperbolic graphs, generalized fine actions defined below and the Bowditch Boundary of a relatively hyperbolic group.

Suppose that $(K,\mc{D})$ acts relatively geometrically on a CAT(0) cube complex $\tilde{C}$. 
Then \cite[Theorem 5.1]{CharneyCrisp} implies that both $\tilde{C}$ and $\tilde{C}\oskel$ are quasi-isometric to the coned-off Cayley graph. 
Note that $\tilde{C}\oskel$ may have infinite edge stabilizers and therefore fails to be a fine hyperbolic graph with finitely many $K$-orbits of edges and finite edge stabilizers. 

\begin{definition}
Let $\Gamma$ be a graph and let $\mc{B}$ be a collection of pairwise disjoint connected sub-graphs of $\Gamma$. The \textbf{complete electrification of $\Gamma$ with respect to $\mc{B}$} is the graph $\Gamma'$ constructed by collapsing each $B\in\mc{B}$ to a single vertex $v_B$. 
\end{definition}

Recent work of the authors and Groves \cite{GenFine} introduced the following notion in which complete electrifications relate a graph like $\tilde{C}\oskel$ to a fine hyperbolic graph with compact fundamental domain and finite edge stabilizers. 

\begin{definition}
\label{D:GeneralizedFine}
Let $K$ be a group that acts by isometries on a $\delta$--hyperbolic cell complex, and let $\mc{D}$ be a finite and almost malnormal collection of finitely generated subgroups of $K$. 
For each $D\in\mc{D}$ and $k\in K$, let $\Sigma_{D^k}$ be the sub-graph consisting of cells whose stabilizers are \textbf{commensurable} to $D^k$; that is, those cells whose stabilizer is finite index in $D^k$. 
A \textbf{circuit without peripheral backtracking} is a circuit whose intersection with any $\Sigma_{D^k}$ is connected. 
We say that $\Sigma$ is \textbf{generalized fine} with respect to the action of $(K,\mc{D})$ if:
\begin{enumerate}
\item the quotient $\leftQ{\Sigma}{K}$ is compact,
\item every cell with infinite stabilizer lies in some $\Sigma_{D^k}$,
\item the sub-graphs $\Sigma_{D^k}$ are compact and connected, and
\item for every edge $e$ of $\Sigma$ with finite stabilizer and every $n\in\naturals$, there are finitely many circuits without peripheral backtracking of length $n$ that contain $e$. 
\end{enumerate}
\end{definition}

Collapsing the stabilized sub-graphs $\Sigma_{D^k}$ yields a fine hyperbolic graph. 
We will  see that $\EG_{bal}\oskel$ is generalized fine with respect to the action of $(G,\mc{P})$. 
We now introduce some properties of generalized fine actions from \cite[Proposition 3.4 and Proposition 3.5]{GenFine}:
\begin{proposition}
\label{electricQI}
Let $\Sigma$ be generalized fine with respect to the action of $(K,\mc{D})$. The graph $\Sigma'$ obtained by electrifying $\mc{B} =\{\Sigma_{D^k}:\,D\in \mc{D},\,k\in K\}$ is a fine hyperbolic graph with a $K$-cocompact action where all edge stabilizers are finite.
Moreover, the collapse map $\beta:\Sigma\to\Sigma'$ that collapses the $\Sigma_{D^k}$ is a $K$-equivariant quasi-isometry. 
\end{proposition}

Bowditch explains how to construct $\partial_{\mc{D}}K$ from a suitable action on a fine hyperbolic graph \cite[Section 9]{BowditchRH}. We recall this construction and how to extend it to generalized fine actions as in \cite[Definition 4.6]{GenFine}:
\begin{definition}
\label{D: fine boundary}
Let $(K,\mc{D})$ be a relatively hyperbolic pair and suppose that $\Gamma$ is a graph that is generalized fine with respect to the action of $(K,\mc{D})$. 
For each $D\in\mc{D}$ and $k\in K$, let $\Gamma_{D^k}$ be the sub-graph consisting of cells whose stabilizers are commensurable to $D^k$. 
Let $\Gamma'$ be a complete electrification of $\Gamma$ with respect to the $\Gamma_{D^k}$ and for all $D\in\mc{D}$ and $k\in K$, let $v_{D^k}$ be the vertex of $\Gamma'$ stabilized by $D^k$.
Let $\triangle \Gamma' = \partial\Gamma' \sqcup V(\Gamma')$ endowed with the following topology:
If $A$ is any finite subset of the vertices of $\Gamma'$ and $a\in \triangle \Gamma'$, define $N(a,A)$ to be the set of $b\in \triangle \Gamma'$ so that every geodesic from $a$ to $b$ avoids $A\setminus a$. A subset $U\subseteq \triangle \Gamma'$ is open if for every $a\in U$, there exists a finite set of vertices $A\subseteq V(\Gamma')$ so that $N(a,A)\subseteq U$.

Let $\Pi_{\Gamma'} = \{v_{D^k}:\,D\in\mc{D},\,k\in K\}$. 

The \textbf{Bowditch Boundary} of $(K, \mc D)$ is $\partial_{\mc D}K = \partial \Gamma'\cup \Pi_{\Gamma'}$ where $\partial \Gamma'$ is the visual boundary of the hyperbolic graph $\Gamma'$. 
The topology on $\partial_{\mc D}K$ is the subspace topology induced by the topology on $\triangle \Gamma'$. 
\end{definition}  

\subsection{Hypersets in Graphs}\label{S: separation}

Let $\Gamma$ be a connected graph.
A \textbf{hyperset} $L$ in $\Gamma$ is a collection of edge midpoints and vertices of $\Gamma$ so that $\Gamma\setminus L$ has two components. 
We will be mostly concerned with sets that arise as intersections of the walls from \Cref{S: wall construction} with the one-skeleton of their ambient space.
A \textbf{hyperset carrier} $J$ is a minimal connected induced sub-graph of $\Gamma$ containing $L$ 
having the property that
the intersection of $J$ with each complementary component of $\Gamma \smallsetminus L$ is connected.

Here are the two natural ways that hypersets and carriers will arise for us:
\begin{example}
Let $(K,\mc{D})$ act relatively geometrically on a CAT(0) cube complex $\tilde{C}$, and let $W$ be a hyperplane. 
Then $L = W\cap \tilde{C}^{(1)}$ is a hyperset. 
A hypercarrier $J$ for $L$ is the intersection of the carrier of $W$ with $\tilde{C}^{(1)}$.  	
\end{example}
In this situation, we will refer to $L$ as the \textbf{hyperset associated to $W$} and $J$ as \textbf{the hypercarrier} of the hyperset associated to $W$.

\begin{example}
Suppose $\Omega$ is a complex consisting of polygons with an even number of sides and cubes. Let $W$ be an $\Omega-$hyperstructure that separates $\Omega$ into two complementary components, or if $\Omega = X_{bal}$ as defined above, then $W$ can be also be a projected hypergraph. Then $L = W\cap \Omega^{(1)}$ is a hyperset. A hypercarrier $J$ for $L$ is the intersection of the associated hypercarrier for $W$ with $\Omega$.
\end{example}
As in the cube complex case, we will refer to $L$ as the \textbf{hyperset associated to $W$} and $J$ as \textbf{the hypercarrier} of the hyperset associated to $W$.
Let $\Sigma$ be a $\delta$-hyperbolic graph that is generalized fine with respect to the action of $(K,\mc{D})$ and for each $D\in\mc{D}$ and $k\in K$, let $\Sigma_{D^k}$ be the compact sub-graph stabilized by $D^k$. 

Hypersets are preserved when electrifying stabilized sub-graphs of a generalized fine action. 
\begin{proposition}\label{P: collapsing hypersets}
Let $\Sigma$ be generalized fine with respect to the action of $(K,\mc{D})$ and let $\Sigma_{D^k}$ be the sub-graph of cells whose stabilizer is commensurable to $D^k$ for $D\in\mc{D}$, $k\in K$. 
Let $\sigma:\Sigma\to \Gamma$ be the map that collapses the $\Sigma_{D^k}$. 
Let $S = \bigcup \{\Sigma_{D^k}:\, \Sigma_{D^k}\cap L \ne \emptyset\}$. 
If $L$ is a hyperset and $J$ is a quasiconvex hyperset carrier, then $\sigma(L)$ is a hyperset in $\Sigma(D^k)$, $\sigma(J)$ is a hyperset carrier and the components of $\Gamma\setminus \sigma(L)$ are images of the components of $\Gamma\setminus (S\cup L)$.   
\end{proposition}

\subsection{A Separation Criterion for the Bowditch Boundary}

In this section, we recall a criterion from \cite{GenFine} for separating points in the Bowditch Boundary constructed from a fine hyperbolic graph. 

We set the following assumptions for the remainder of this subsection:
\begin{hypotheses}\label{separation setup}
Let $(K,\mc{D})$ be a relatively hyperbolic pair and let $K$ act on a fine $\delta$--hyperbolic graph $\Gamma$ with the following properties:
\begin{itemize}
\item The action of $K$ is cocompact,
\item edge stabilizers are finite, and
\item each $D\in\mc{D}$ stabilizes a single vertex. 
\end{itemize}
Let $L$ be a hyperset with connected quasiconvex carrier $J$. 
\end{hypotheses}

\begin{definition}\label{D: two-sided carrier}
With the setup in \Cref{separation setup}, $J$ has the \textbf{two-sided carrier property} if there exist connected quasiconvex subsets $J^+$ and $J^-$ so that $J = J^+\cup J^-$, $J^+\cap J^-\subseteq L$, every path between vertices in distinct components of $\Gamma\setminus L$ must intersect both $J^+$ and $J^-$ and if $v$ is a vertex of $J^-\cap J^+$ with infinite stabilizer, then $v\in \Lambda\Stab_K(L)$.  
\end{definition}

When \Cref{separation setup} arise from electrifying the stabilized sub-graphs of $(K,\mc{D})$ acting relatively geometrically on a CAT(0) cube complex $\tilde{X}$, the two-sidedness of the hyperplanes gives us a natural way to show that hypersets in $\Gamma$ have the two-sided carrier property, see \cite[Proposition 6.3]{GenFine}.

When the Bowditch Boundary is constructed as in \Cref{D: fine boundary} from a fine hyperbolic graph, the points of the Bowditch Boundary are either \textbf{conical limit points} that lie in $\partial \Gamma$, the visual boundary of $\Gamma$ or are \textbf{parabolic vertices} of $\Gamma$ that are stabilized by maximal parabolics. 

Note that $\Gamma$ is not a proper metric space, so existence of a bi-infinite geodesic connecting points in the visual boundary of $X_{bal}\oskel$ does not follow immediately from the standard techniques for constructing such a bi-infinite geodesic.

\begin{proposition}
\label{biinfinite-geod}
Let $x,y\in \partial_{\mc{D}}K$. There exists a bi-infinite combinatorial geodesic 
between $x$ and $y$ in $\Gamma$. 
\end{proposition}

\begin{proof}
There is a proper hyperbolic space $Z$ so that $\partial Z \cong \partial_{\mc{D}}K$. 
Let $\triangle \Gamma = \Gamma^{(1)}\cup \partial \Gamma$.  
By \cite[Lemma 9.1]{BowditchRH}, there exists a map from $\gamma:\triangle \Gamma \to Z$ that induces a homeomorphism $\bar\gamma:\partial_{\mc{D}}K \to \partial Z$. Since $Z$ is proper and hyperbolic, there exists a bi-infinite geodesic $\alpha$ in $Z$ that joins $\bar{\gamma}(x)$ and $\bar{\gamma}(y)$. 
By \cite[Lemma 9.3]{BowditchRH}, there exists a bi-infinite geodesic $\Upsilon$ in $\Gamma$ so that $\gamma(\Upsilon)$ lies a bounded distance from $\alpha$, so $\Upsilon$ joins $x$ to $y$ in $\partial_{\mc{D}}K$. 
\end{proof}

\begin{definition}\label{D: separation}
Assuming \Cref{separation setup}, we say that $L$ \textbf{separates $p, q$} if $p, q\notin \Lambda \Stab_K(L)$ and one of the following holds:
\begin{itemize}
\item $p, q$ are both conical limit points, and there exists some geodesic $\gamma:(-\infty,\infty)\to \Gamma$ with $\lim_{t\to \infty}\gamma(t) = p$ and $\lim_{t\to-\infty}\gamma(t) = q$ so that there exists $T>0$ so that for all $t_- < -T <0 <T<t_+$, $\gamma(t_+)$ and $\gamma(t_-)$ are in distinct components of $\Gamma\,\setminus\, L$, 
\item $q$ is a parabolic vertex in $\Gamma$, $p$ is a conical limit point, and there exists some geodesic $\gamma:[0,\infty)\to \Gamma$ from $q=\gamma(0)$ to $p$ so that for $t$ sufficiently large, $\gamma(t)$ and $q$ are in distinct components of $\Gamma\,\setminus\, L$,
\item $p, q$ are both parabolic vertices in $\Gamma$, and $p, q$ lie in distinct components of $\Gamma\,\setminus\, L$. 
\end{itemize}
\end{definition} 

For the next theorem, we need a quick definition:
\begin{definition}
Let $G$ act on the components of a topological space and let $H\le G$. Two components $C_1,C_2$ are $H$--distinct if $h\cdot C_1\ne C_2$ for all $h\in H$. 
\end{definition}

In \cite[Theorem 5.7]{GenFine}, Groves and the authors prove that separation in the sense of \Cref{D: separation} implies separation in the Bowditch Boundary: 
\begin{theorem}[{\cite[Theorem 5.7]{GenFine}}]
\label{T: separation criterion}
Assume \Cref{separation setup} and assume the setup satisfies the two-sided carrier property. If $L$ separates $p, q\in \partial_{\mc{D}}K\setminus \Lambda\Stab_K(L)$, then there exists a finite index subgroup $H_L\le \Stab_K(L)$ so that $p$ and $q$ are in $H_L$--distinct components of $\partial_{\mc{D}}K\setminus \Lambda H_L$. 
\end{theorem}

\section{The Polygonal Case}
\label{S: polygonal case}

\subsection{Separating boundary points in a polygonal complex when the hyperstructures are trees} \label{S: polygonal case 1}

For Section~\ref{S: polygonal case 1}, we set these hypotheses:
\begin{hypotheses}\label{H: general polygonal setup}
Let $X_0$ be a polygonal complex where each polygon has an even number of sides. Let $(K,\mc{D})$ be a relatively hyperbolic pair that acts cocompactly on $X_0$ so that every edge has finite stabilizer, every $D\in\mc{D}$ stabilizes exactly one vertex and every infinite cell stabilizer is of the form $D^k$ for some $D\in \mc{D}$ and $k\in K$.
\end{hypotheses}

As in the more specific case of when $X_0 = X_{bal}$, given an edge $e$ of $X_0$, we let $Z_e^X$ denote the $X_0$--hyperstructure dual to $e$. We also refer to $Z_e^X$ as a hypergraph. 

We now impose some additional conditions that we will subsequently prove are satisfied by $X_{bal}$:
\begin{hypotheses}\label{H: polygonal wall hypotheses}
In addition to Hypotheses~\ref{H: general polygonal setup}, we assume that:
\begin{enumerate}
\item For any edge $e$ of $X_0$, $Z_e^X$ is an embedded tree in $X_0$ whose intersection with $X_0\oskel$ is quasi-convex in $X_0\oskel$, \label{I: is a tree}
\item For any edge $e$ of $X_0$, $Z_e^X$ separates $X_0$ into two complementary components. \label{I: separating 2} 
\item Further, $\Stab_K(Z_e^X)$ has finite intersection with any vertex stabilizer and is therefore full.  \label{I: strong full}
\item If $\gamma$ is a combinatorial geodesic, there exists an edge $e$ of $\gamma$ so that $Z_e^X$ crosses $\gamma$ exactly once. If $\gamma$ is infinite, there exists an $N\in\naturals$ so that every subsegment of $\gamma$ with length $N$ contains an edge $e$ so that $Z_e^X$ intersects $\gamma$ exactly once.  \label{I: many good walls exist}
\end{enumerate}
\end{hypotheses}
If $X_0$ satisfies these properties, then we say $X_0$ has \textbf{suitable walls}. 

We make a quick observation so we can use results from \Cref{S: fine genfine}. 
\begin{proposition}
The graph $X_0\oskel$ is fine and hyperbolic. 
\end{proposition}

\begin{proof}
The graph $X_0\oskel$ is hyperbolic by \Cref{CCrestatement}. Since the edge stabilizers are finite, each maximal parabolic stabilizes exactly one vertex and the action of $G$ is cocompact, so by \cite[Lemma 4.5]{BowditchRH}, $X_0\oskel$ is fine.  
\end{proof}

\begin{proposition}\label{P: cocompactness on polygonal hyp}
$\Stab_K(Z_e^X)$ acts cocompactly on $Z_e^X$, and $\Stab_K(Z_e^X)$ is relatively quasiconvex in $(K,\mc{D})$. 
\end{proposition}

\begin{proof}
Cocompactness follows immediately from \Cref{L: hyperstructure stab cocompact}. Relative quasiconvexity then follows by \Cref{H: polygonal wall hypotheses}\eqref{I: is a tree} and \Cref{T: qccrit}. 
\end{proof}

If $Z_e^X$ is a hypergraph in $X_0$, then the associated hyperset is $Z_e^X\cap X_0\oskel$. 
Similarly, if $Y_e^X$ is the carrier of $Z_e^X$, then the associated hyperset carrier is $Y_e^X\cap X_0\oskel$, which we denote by $J$ for ease of notation. 
We say that a hypergraph $Z_e^X$ \textbf{separates} $p,q\in \partial_{\mc{D}}K$ if the associated hyperset $L$ separates $p,q$ in $\partial_{\mc{D}}K\setminus \Lambda \Stab_K (Z_e^X)$ in the sense of \Cref{D: separation}. 

We recall some elementary facts about $\delta$--hyperbolic spaces.  



\begin{lemma}[Uniform fellow-traveling]
\label{unif-fellow-travel}
Let $X_0$ be a $\delta$--hyperbolic space.  Let $U \subseteq X_0$ be any convex subset.  
There exists $R = R(\delta)$ such that if $\gamma$ is a geodesic so that $\gamma \subseteq {\mc N}_\delta(U)$ then $\gamma \subseteq {\mc N}_R(U)$.  
\end{lemma}

\begin{proof}
This is an immediate consequence of the Morse lemma for hyperbolic spaces.  We can view $\gamma$ as a $(1, 2\delta)$--quasi-geodesic joining pairs of points in $U$.  
\end{proof}

\begin{definition}
Let $\gamma \subseteq X_0\oskel$ be a (possibly finite) geodesic so that any endpoints are vertices.  We define the \textbf{skewer set} as follows
\[
\sk(\gamma) = \{ Z_e^X \, \mid \, e \text{ is an edge of  } \gamma \}
\] 
\end{definition}

We say that each $Z_e^X \in \sk(\gamma)$ is \textbf{dual} to $\gamma$. 

\begin{remark}\label{properties-skewer}
If $\gamma$ is infinite and Hypotheses~\Cref{H: polygonal wall hypotheses} hold, there are infinitely many edges in $\gamma$ whose dual hypergraph crosses $\gamma$ exactly once. Therefore, $\sk(\gamma)$ is infinite when $\gamma$ is a geodesic ray or bi-infinite geodesic. There are finitely many $K$-orbits of edges in $X_0$. If $\gamma$ is infinite, there is a hypergraph $Z$ with infinitely many $K$-translates in $\sk(\gamma)$ that cross $\gamma$ exactly once which we will denote:
\[
\mc S_\gamma := \{ g_i Z \}_{i = 0}^\infty.
\]
Without loss of generality we can and will assume $g_0 = id_K$.
Moreover, each hypergraph $O \in \sk(\gamma)$ disconnects $\gamma$ into two components such that each lie in distinct components of $X_0\oskel \smallsetminus O$. 
\end{remark} 

Our goal in describing skewer sets is that we may understand stabilizers of hypergraphs using acylindricity of the $G$--action on $X_0$ from the relative hyperbolicity of $(K,\mc{D})$:
\begin{proposition}
\label{P: acylindrical}
Let $X_0$ and $(K, \mc{D})$ be as in \Cref{H: general polygonal setup}.
The action of $K$ on $X_0\oskel$ is acylindrical.
\end{proposition}

\begin{proof}
\cite[Theorem 5.1]{CharneyCrisp} implies the orbit map induces a $K$--equivariant quasi-isometry from any relative Cayley graph $\Cay(K, \mc{D}, S)$ with $S$ finite to $X_0\oskel$. 
The action of $K$ on $\Cay(K, \mc{D}, S)$ is acylindrical \cite[Proposition~5.2]{OsinAH}.  

Hence, the action of $K$ on $X_0\oskel$ is acylindrical because acylindricity is preserved under equivariant quasi-isometry.
\end{proof}

Note that in the definition of an action being acylindrical along a subgroup \Cref{def:AQI-cobddSubset} requires the subgroup $H$ to be quasi-isometrically embedded in $K$ with respect to any word metric associated to a finite generating set.

\begin{lemma}
The orbit of $\Stab_K (Z_e^X)$ is quasi-isometrically embedded in $X_0\oskel$.  
\end{lemma}

\begin{proof}
By fullness, the intersection of $\Stab_K(Z_e^X)$ with any cell stabilizer of $X_0\oskel$ is finite. 
Therefore, $\Stab_K (Z_e^X)$ acts properly on the tree $Z_e^X$. 
By \Cref{P: cocompactness on polygonal hyp}, $\Stab_K (Z_e^X)$ acts cocompactly on $Z_e^X$. 
Therefore, distance in $\Stab_K (Z_e^X)$ is coarsely equivalent to distance in $Z_e^X$, which is in turn coarsely equivalent to distance in $X_0\oskel$ because $Z_e^X\cap X_0\oskel$ is quasiconvex in $X_0\oskel$.
\end{proof}

Recall that the Bowditch Boundary $\partial_{\mc{D}} K$ can be partitioned into parabolic points and conical limit points as in \Cref{D: fine boundary}.  Moreover, since $X_0\oskel$ is a fine hyperbolic graph that witnesses the relative hyperbolicity of $(K,\mc{D})$, parabolic points can be identified with vertices in $X_0\oskel$ and conical limit points can be viewed as points in the visual boundary $\visualBoundary X_{0}\oskel$.  

\begin{proposition}\label{P: hypergraph separation}
Let $p,q\in \partial_{\mc{D}} K$ be a pair of distinct points. Then there exists a hypergraph $Z_e^X$ whose associated hyperset separates $p$ and $q$. 
\end{proposition}

\begin{proof}
The partition of the Bowditch Boundary into parabolic and conical limit points splits our consideration into the following three cases: (1) both points are parabolic, (2) both points are conical limit points, and (3) one point is conical and the other is parabolic.

\textbf{Case 1: $p,q$ are both parabolic points. }

Let $\gamma$ be a geodesic segment in $X_0\oskel$ joining $p$ and $q$. 
By Hypotheses~\ref{H: polygonal wall hypotheses}\eqref{I: many good walls exist}, there exists an edge $e_0$ so that the dual hypergraph $Z_{e_0}^X$ crosses $\gamma$ exactly once. 
Let $H=\Stab_K (Z_{e_0}^X)$.  

In this case, let $Z_e^X = Z_{e_0}^X$. 
Since $Z_{e}^X$ splits $X_{0}\oskel\setminus Z_e^X$ into two components by Hypotheses~\ref{H: polygonal wall hypotheses}\eqref{I: separating 2}, it suffices to show that $p,q\notin\Lambda H$. 
By Hypotheses~\ref{H: polygonal wall hypotheses}\eqref{I: strong full} the intersection of the parabolic stabilizer of $p$ or $q$ with $H$ is finite.
Therefore, both $p$ and $q$ are not contained in the limit set of $H$. 

\textbf{Case 2: $p,q$ are both conical limit points.}

By \Cref{biinfinite-geod} there is a geodesic $\gamma$ in $X_{0}\oskel$ joining $p,q$. 
Let $\mc{S}_\gamma$ be the subset of $\sk(\gamma)$ described in \Cref{properties-skewer}. 
Fix $ R  = R(\delta)$, the uniform fellow-traveling constant of \Cref{unif-fellow-travel}.
We claim that all but finitely many $Z_f^X\in \mc{S}_\gamma$ have $p\notin \Lambda \Stab_K(Z_{f}^X)$. 
Suppose not. Hypotheses~\ref{H: polygonal wall hypotheses}\eqref{I: many good walls exist} implies that there are infinitely many edges whose dual hypergraph crosses $\gamma$ exactly once, and there are finitely many edge orbits. Therefore we may choose $g_1,g_2,\ldots$ elements of $K$ so that $g_1e_0 ,g_2e_0 ,\ldots$ is an infinite collection of edges of $\gamma$ so that 
\begin{enumerate}
\item $Z_{g_ie_0}^X\in \mc{S}_{\gamma}$, and
\item if $H_i = \Stab_K (Z_{g_ie_0}^X)$, then $p\in \Lambda H_i$. 
\end{enumerate}
The $Z_{g_ie_0}^X$ are distinct because they each cross $\gamma$ exactly once. 
Therefore, $\{g_iH\}$ is an infinite collection of cosets of $H$ because the orbit of $g_iH$ intersects $g_ie$.
Parameterize $\gamma:(-\infty,\infty)\to X_{0}\oskel$ so that $\lim_{t\to\infty}\gamma(t) = p$. 
Then there exists $t_i$ so that if $t>t_i$, $\gamma(t) \subseteq \mc{N}_R(Z_{g_ie_0}^X)$. 
Therefore for any $D\ge 0$:
\[ \# \left\{k:\, \diam\left(\bigcap_{i=0}^k \mc{N}_{2R}(Z_{g_ie_0}^X)\right) >D\right\}\]
is unbounded. Since $Z_{e_0}^X$ is $H$-cocompact, the image of $g_iH$ under the orbit map is coarsely $Z_{g_ie_0}^X$.
Thus we obtain a contradiction to the fact that the action of $K$ is acylindrical in the sense of \Cref{def:AQI-cobddSubset}.

Similarly, all but finitely many $Z_f^X\in \mc{S}_\gamma$ do not have $q\in \Lambda \Stab_K(Z_f^X)$. 

Therefore there exists some $Z_e^X\in \mc{S}_\gamma$ so that $p,q\notin \Lambda\Stab_K(Z_e^X)$. 
Since $\gamma$ can only cross $Z_e^X$ once, we see that $p$ and $q$ are separated by $Z_e^X$. 

\textbf{Case 3: $p$ is a conical limit point and $q$ is a parabolic vertex}. 

There exists a geodesic ray $\gamma$ from $q = \gamma(0)$ to $p$ by the definition of the visual boundary.  
By following the same argument in the previous case, there exists an $e$ dual to $\gamma$ so that $p\notin \Lambda\Stab_K(Z_e^X)$ and $Z_e^X$ crosses $\gamma$ exactly once.
The same argument from the first case tells us that $q\notin \Lambda \Stab_K(Z_e^X)$. 
Since $\gamma$ crosses $Z_e^X$ exactly once and $Z_e^X$ separates $X_{0}\oskel$ into two complementary components, $Z_e^X$ separates $p$ and $q$. 
\end{proof}

\begin{proposition}\label{P: hypergraph two-sided carrier}
If $J$ is the associated hyperset carrier of $Z_e^X$, then $J$ has the two-sided carrier property.
\end{proposition}

\begin{proof}
By Hypotheses~\ref{H: polygonal wall hypotheses}\eqref{I: separating 2}, $Z_e^X$ separates $X_{0}$ into two distinct components and therefore $L = Z_e^X \cap X_0\oskel$, the hyperset associated to $Z_e^X$,   naturally separates $X_{0}\oskel$ and $J$ decomposes into two distinct connected pieces $J^+$ and $J^-$ so that $J^+\cap J^- = L$ and $J^+\setminus L,\,J^-\setminus L$ lie in distinct components of $X_{0}\oskel\setminus L$. 
A combinatorial path from one component of $X_{0}\setminus Z_e^X$ must pass through an edge dual to $Z_e^X$, so every path between $J^+$ and $J^-$ must pass through $L$. 
Recall that $J$ is the convex one-skeleton of the carrier of $Z_e^X$ and the cocompactness of the action of $K$ on $X_{0}$ gives a bound $M$ on the number of sides of any polygon in $X_{0}$, so $J^+$ and $J^-$ are $M$--quasiconvex in $X_{0}\oskel$. By construction $J^+\cap J^-= L$. 
\end{proof}


\begin{theorem}\label{T: hypergraph separation}
Let $p,q\in\partial_{\mc{D}}K$ be distinct, then there exists a hypergraph $Z_e^X$ and a finite index full relatively quasiconvex $K_W\le \Stab_G (Z_e^X)$ so that $p$ and $q$ are in $K_W$--distinct components of $\partial_{\mc{D}}K \setminus\Lambda K_W$. 
\end{theorem}

\begin{proof}
By \Cref{P: hypergraph separation} there exists a hypergraph that separates $p$ and $q$. The action of $G$ on $X_0\oskel$ has trivial edge stabilizers, the action is cocompact and every maximal parabolic fixes a vertex.  Therefore, by \Cref{T: separation criterion}, there exists a finite index subgroup $K_W$ of $\Stab_G (Z_e^X)$ so that $p$ and $q$ lie in $K_W$-distinct components of $\partial_{\mc{D}}K\setminus \Lambda K_W$. Full relative quasiconvexity is preserved by taking finite index subgroups, so full relative quasiconvexity of $K_W$ is implied by Hypotheses~\ref{H: polygonal wall hypotheses}. 
\end{proof}

\section{Why $X_{bal}$ has suitable walls}
\label{S: X hypotheses satisfied}


We easily obtain Hypotheses~\ref{H: polygonal wall hypotheses} Items~\eqref{I: is a tree} and \eqref{I: separating 2}. 
Each hypergraph $Z_e^X$ is an embedded tree whose intersection with $X_0\oskel$ is quasi-convex in $X_0\oskel$ by \cite[Theorem~3.10]{MartinSteenbock} showing \eqref{I: is a tree}.   

Moreover, \cite[Proposition~3.14]{MartinSteenbock} shows that each $Z_e^X$ has exactly two complementary components proving (\ref{I: separating 2}).

\subsection{Fullness}
In this subsection, we establish Hypothesis~\ref{H: polygonal wall hypotheses}\eqref{I: strong full} and some more general results that will help us in the blowup case below. 

We first establish that hyperplane stabilizers for relatively geometric actions are full. Proposition~\ref{hyp stab full} will help motivate our technique for proving Proposition~\ref{P: protofull} below and will be used later. 
\begin{proposition}\label{hyp stab full}
Let $(K,\mc{D})$ act relatively geometrically on a CAT(0) cube complex $\tilde{X}$. If $L\le K$ is a hyperplane stabilizer, then $L$ is full. 
\end{proposition}

\begin{proof}
Let $W$ be the hyperplane of $\tilde{X}$ stabilized by $L$. Suppose that there exist $k\in K$ and $D\in\mc{D}$ so that $L\cap D^k$ is infinite. 
$K$ acts relatively geometrically, so every maximal parabolic stabilizes a vertex and every cell stabilizer is finite index in a maximal parabolic. 
Then there exists a finite index subgroup $D_v\le D^k$ that stabilizes a vertex so that $D_v\cap L$ is infinite.

\textbf{Claim: an infinite subgroup of $D^k$ fixes an edge $e$ of $\tilde{X}$ dual to $W$.} 
By convexity of $W$, the nearest point projection of $v$ to $W$ with respect to the $\CAT(0)$ metric is well-defined and is a unique point.  Let  $p_0\in W$ denote this point.  Since $p_0\in W$, $p_0$ lies in the interior of some cube $M$ of $\tilde{X}$ that has an edge dual to $W$.
Since $K$ acts on $\tilde{X}$ by isometries, the uniqueness of $p_0$ implies that $(D_v\cap L)$ fixes $p_0$. 
Since $\Stab_K(p_0)$ fixes the cube $M$ setwise, a finite index subgroup $D_1\le \Stab_K(p_0)$ fixes $M$ pointwise. 
Then $D_M = D_1\cap D_v \cap L$ is an infinite subgroup of $D_v\cap L\le D^k$ that fixes $M$ and fixes an edge $e$ dual to $W$. This completes the proof of the claim. 

\textbf{Claim: the stabilizer of $e$ is finite index in $D^k$.} Since $\Stab_K(e)$ is infinite, $\Stab_K(e)$ is finite index in a maximal parabolic. By construction, $\Stab_K(e)$  has infinite intersection with $D^k$. Thus $\Stab_K(e)$ is finite index in $D^k$ because distinct maximal parabolics have finite intersections. This completes the proof of the second claim.

Proposition~\ref{hyp stab full} now follows because the $K$--stabilizer of $e$ is contained in $L$.  
\end{proof}

\begin{figure}[h]
\begin{tikzcd}
D^k \cap L
&	D_v \arrow[r, "f.i."]
&	D^k \\
\Stab_P(v) \cap L \arrow[d, hook] \arrow[u, "f.i."] \arrow[ur, hook]
&	D_v\cap L \arrow[u, hook]
&	L \arrow[u, "f.i."] \\
\Stab_K(p_0)
& D_M \arrow[ul, "f.i."] \arrow[l,hook] \arrow[u, hook] \arrow[ur, hook] 
\arrow[r, hook]
&  \Stab_K(f) \arrow[u, hook] 
\end{tikzcd}
\caption{Commutative diagram depicting the various subgroups considered in the proof of \Cref{hyp stab full}.}
\end{figure}

%
%


The following Proposition now implies Hypotheses~\ref{H: polygonal wall hypotheses}\eqref{I: strong full}:

\begin{proposition}\label{P: protofull}
Let $Z$ be a hypergraph in $X_{bal}$ with stabilizer $H$ and let $v$ be a vertex of $X_{bal}$ with stabilizer $G_v$. 
If $v$ is not contained in $Z$, then $G_v\cap H$ is finite. 
\end{proposition}

Recall that as a hypergraph, $Z$ can be either an $X_{bal}$--hyperstructure or a projected hypergraph, the image of an $\EG_{bal}-$hyperstructure under the natural projection $\EG_{bal}\to X_{bal}$. We first prove an auxiliary Lemma:

\begin{lemma}\label{L: hypcar mindist}
Let $Y$ be the hypercarrier of $Z$ in $X_{bal}$, 
let $v$ be a vertex of $X_{bal}$, and 
let $D = \min\{d_{X_{bal}^{(1)}}(z,v):\,z\in Y^{(0) }\}$. Then $S:=\{z\in Y^{(0)}:\,d_{X_{bal}^{(1)}}(z,v) = D\}$ is finite.
\end{lemma}

\begin{proof}
We proceed by induction on $D$. When $D = 0$, the statement is obvious. 

Suppose toward a contradiction that $S$ is infinite. 
Choose a distinguished vertex $x\in S$ and	$\gamma$ a geodesic path in $X_{bal}^{(1)}$ from $v$ to $x$. 
Let $S \smallsetminus \{x\} = \{x_1,x_2, \dots \}$ be an enumeration\footnote{Recall $X_{bal}^{(1)}$ is a subdivided quotient of a tree with countably many vertices and therefore has countably many vertices.} of vertices distinct from $x$. 
Choose geodesic paths $\gamma_i$ in $X_{bal}^{(1)}$ from $v$ to $x_i$ with the following property: if $v_1$ and $v_2$ are vertices of $\gamma_i$ that are also vertices of $\gamma$, then $\gamma_i$ and $\gamma$ follow the same path between $v_1$ and $v_2$. 
If all but finitely many $\gamma_i$ leave $v$ by following the same edge, then there exists a vertex $v'$ so that $d_{X_{bal}^{(1)}}(v',Y^{(0)}) = D-1$ and there are infinitely many vertices in $Y$ that realize the minimal distance $D-1$. 
By induction, we obtain a contradiction.

Therefore, by possibly passing to an infinite subset of the $\gamma_i$, we may assume the $\gamma_i$ only intersect at $v$. 
Note that $\gamma_i\cap Y=\{x_i\}$ and $\gamma\cap Y = \{x\}$ because $x,x_i\in S$. 
Then there exists a path $\sigma_i\subseteq Y\cap X_{bal}^{(1)}$ from $x$ to $x_i$ of length at most $2D$ because $Y^{(1)}$ is convex. 
Then the paths $\rho_i = \gamma_i\cup \sigma \cup \gamma$ are circuits based at $v$ in $X_{bal}^{(1)}$ of length at most $4D$ that include the edges of $\gamma$. There are infinitely many $\rho_i$ which violates the fineness of $X_{bal}^{(1)}$. 
\end{proof}

\begin{proof}[Proof of \Cref{P: protofull}]
Recall that $Y$ denotes the hypercarrier of $Z$ and $H = \Stab_{G}(Z)$. Let $H_v := G_v \cap H$.  First, assume that $v\notin Y^{(0)}$. Then there exists a vertex $x\ne v$ so that $x\in Y^{(0)}$ minimizes the (1--skeleton) distance from $Y$ to $v$ by \Cref{L: hypcar mindist}. 

We claim that $H_v$ is commensurable to a subgroup of $G_v \cap G_x$.  Indeed, by \Cref{L: hypcar mindist} the $G_v$ orbit of $x$ is finite.  Hence, $H_v$ has a finite index subgroup which stabilizes $x$.  Stabilizers of distinct vertices in $X_{bal}$ have trivial intersection because the local group at each edge is trivial. Hence, $H_v$ is finite.  


It remains to show the case where $v\in Y^{(0)}$. Since $Z$ does not contain $v$, there exists a polygon $U\subseteq Y$ so that $v \in U$ and the interior of $U$ intersects $Z$. 

If $U$ is the unique polygon of $Y$ containing $v$ then $H_v$ stabilizes $U$.  
Moreover, $H_v$ acts on $U$ as a subgroup of a finite dihedral group, so a finite index subgroup fixes $U$ pointwise.  
In particular, this finite index subgroup fixes an edge, so $H_v$ is finite because edge stabilizers are trivial.

Hence, suppose there are multiple polygons of $Y$ that contain $v$. Let
\[\mc{U} = \{U:\,\text{U is a polygon of $Y$ that contains $v$}\}.\]

\textbf{Claim: if $U\in \mc{U}$, then $\partial U$ has at most two edges dual to $Z$.} The abstract carrier of $Z$ embeds by Proposition~\ref{projected hypercarrier embed}. The abstract carrier has one copy of $U$ for every pair of diametrically opposed edges dual to $Z$. Since the natural map from the abstract carrier to $Y$ is injective on the interior of polygons, there can only be one polygon of the abstract carrier whose interior maps to the interior of $U$. 

\textbf{Claim: there is a path $\sigma$ from $v$ to $Z$ so that if $U\in \mc{U}$, then $\partial U$ contains $\sigma$.} 
Let $U_1,U_2\in \mc{U}$ with $U_1\ne U_2$. Since the abstract carrier of $Z$ embeds in $X_{bal}$, the copies of $U_1,U_2$ in the abstract carrier must intersect because $v\in U_1\cap U_2.$ The construction of the abstract carrier (see \Cref{D:hyperstructure}) requires that if the copies of $U_1,U_2$ intersect, then $U_1,U_2$ must share an edge $e$ dual to $Z$ and a path $\sigma$ without backtracking from $e$ to $v$ lies in $\partial U_1\cap \partial U_2$. 

Since $U_1$ may have two edges dual to $Z$, there at most two possibilities for $\sigma$. However, $\sigma$ is a piece, so $|\sigma|\le \frac16 |\partial U_1|$ by the $C'(1/6)$ condition. Therefore, there is only one choice for $\sigma$. 

\textbf{Claim: if $h\in H_v$, then $h\cdot \sigma = \sigma$.} Recall that the action of $G$ on $X_{bal}$ takes polygons to polygons. 	Observe that $H_v = G_v\cap H$ fixes $Z$ setwise and fixes $v$. Therefore, if $U\in \mc{U}$, then $H_v\cdot U \in \mc{U}$. Since $H_v$ fixes $Z$ and $v$, $h\cdot \sigma$ is a path from $v$ to $Z$ of length at most $\frac16|\partial U|$. Since $U$ has two edges dual to $Z$, $h\cdot \sigma$ must be $\sigma$. This completes the proof of the claim.

Finally, since any $h\in H_v$ must fix $\sigma$ pointwise and $v\notin Z$, $H_v$ fixes an edge of $X_{bal}$ pointwise. However, edge stabilizers of $X_{bal}$ are trivial, so we conclude $H_v$ is trivial.  
%
%
\end{proof}

Recall that a fiber complex is the preimage of a vertex of $X_{bal}$ under the projection. 
\begin{lemma}\label{L: hypergraph vertex stabilizer}
Let $H = \Stab_G(W_e^{\EG})$, let $G_v$ be the stabilizer of a vertex $v$ of $X$ and let $EG_v$ be the associated fiber complex. If $W_e^{\EG}$ intersects $EG_v$, then:
\begin{enumerate}
\item $W_e^{\EG}\cap EG_v$ is a hyperplane.
\item $H\cap G_v$ is the stabilizer of a hyperplane of $EG_v$. 
\end{enumerate}
\end{lemma}

\begin{proof}
For the first claim, one might worry that $W_e^{\EG}$ might intersect $\EG_v$ in more than one hyperplane, but then the projection $W_e^{\EG}\to Z_e^{\EG}$ would contain a loop. By \cite[Lemma 3.32]{MartinSteenbock}, $Z_e^{\EG}$ is an embedded tree in $X_{bal}$. This proves the first claim.

Let $W_v = W_e^{\EG}\cap EG_v$ so that $W_v$ is a hyperplane of $EG_v$. If $h\in H\cap G_v$, then $h\cdot W_v$ is a hyperplane of $EG_v$, so $h\cdot W_v = W_v$. Therefore $H\cap G_v\subseteq \Stab_{G_v}(W_v)$.
Similarly, if $h'\in \Stab_{G_v}(W_v)$, then the intersection of $h'\cdot W_e^{\EG}$ and $EG_v$ is $W_v$. Then $h'\cdot W_e^{\EG} = W_e^{\EG}$ because $h'\cdot W_e^{\EG}$ is an $\EG$--hyperstructure and $W_e^{\EG}$ is the unique hyperstructure that intersects $EG_v$ in the hyperplane $W_v$. 
Therefore, $\Stab_{G_v}(W_v)\subseteq G_v\cap H$. 
\end{proof}

\begin{proposition}\label{wall stab full}
Let $H=\Stab_G(W)$ where $W = Z_e^X$ or $W = W_e^{\EG}$. Then $H$ is full in $(G,\mc{P})$.
\end{proposition}

\begin{proof}
First suppose that $W = Z_e^X.$ Since $W$ contains no vertices, $|G_v\cap H|<\infty$ for all vertices $v$ of $X_{bal}$ by \Cref{P: protofull}. 
Every peripheral subgroup is contained in a fiber group, so $H$ has finite intersection with any peripheral subgroup.

Now suppose $W = W_e^{\EG}$.

If $W$ intersects $EG_v$, then \Cref{L: hypergraph vertex stabilizer} implies that the $G_v\cap H$ stabilizes a hyperplane of $EG_v$. 
Hyperplane stabilizers in $EG_v$ are full (\Cref{hyp stab full}), so for any peripheral subgroup $P \in \mc{P}_{v}$ contained in $G_v$, $H\cap P$ is either finite or finite index in $P$. 

Observe that $H$ stabilizes $Z_e^{\EG}$, a (projected) hypergraph in $X_{bal}$. Recall $Z_e^{\EG}$ intersects a vertex $v$ of $X_{bal}$ if and only if $W$ intersects $EG_v$. Therefore, if $v$ is a vertex of $X_{bal}$ so that $W$ does not intersect $EG_v$ then by \Cref{P: protofull}:
\[|G_v\cap H| \le |G_v \cap \Stab_G (Z_e^{\EG})| <\infty \]
Since every peripheral subgroup is contained in a fiber group, it follows that $H$ is full in $(G,\mc{P})$.
\end{proof}



\subsubsection{Infinite geodesics in $X_{bal}$ have many dual hypergraphs that cross exactly once.}

The main result of this subsection is:
\begin{proposition} \label{P: onecross}
Let $\gamma\subseteq X_{bal}\oskel$ be a combinatorial geodesic in $X_{bal}\oskel$. There exists an edge $e$ of $\gamma_0$ so that the hypergraph $Z_e^X$ intersects $\gamma$ exactly once. Further, if $\gamma$ is infinite, there exists $N\in\naturals$ so that every collection of at least $N$ consecutive edges contains an edge $e$ so that $Z_e^X$ crosses $\gamma$ exactly once. 
\end{proposition}

Before proving \Cref{P: onecross}, we introduce some facts about small cancellation, and prove Lemmas~\ref{L: bad walls follow geodesics} and~\ref{L: bad wall polygon} which will be used in the proof of \Cref{P: onecross}. 

\begin{lemma}\label{L: bad walls follow geodesics}
Let $\gamma$ be a geodesic in $X\oskel$ containing an edge $e$ and let $Z_e^X$ be the dual hypergraph. 
Suppose $\gamma$ is dual to $Z_e^X$ at another edge $e'$ of $\gamma$. Let $\gamma_e^X$ be the unique path in $Z_e^X$ between $e$ and $e'$ that does not backtrack. 
Then there is a sequence of polygons $U_1,U_2,\ldots,U_m$ of  $X_{bal}$ so that:
\begin{enumerate}
\item $U_i$ has two edges dual to $\gamma_e^X$ (equivalently $\gamma_e^X$ intersects the interior of $U_i$),
\item the subpath of $\gamma$ from $e$ to $e'$ including these two edges is contained in $\bigcup_{i=1}^m \partial U_i$, and
\item $\gamma\cap \partial U_i\ne\emptyset$ for all $1\le i\le m$. 
\end{enumerate}
\end{lemma}

\begin{proof}
Following $\gamma_e^X$ from $e$ to $e'$, we label the sequence of polygons that $\gamma_e^X$ passes through the interiors of as $U_1,U_2,\ldots,U_n$ (so that each $U_i$ has two edges dual to $\gamma$). 
The collection $U_1,U_2,\ldots,U_n$ are a \textbf{gallery} in the sense of \cite[Definition 3.7]{MartinSteenbock} (where the \textbf{doors} are the edges dual to $\gamma_e^X$). The union of the $U_i$ in $X_{bal}$ is the image of the associated hypercarrier in the sense of \cite[Definition 3.8]{MartinSteenbock}. 
Then \cite[Proposition A.16]{MartinSteenbock} implies that $\gamma$ lies in $\bigcup_i \partial U_i$. 

Let $e = e_0,e_1,\ldots,e_k=e'$ be the edges of $\gamma$ in the interval between and including $e,e'$. 

\textbf{Claim: if $U,V$ are distinct polygons of the carrier of $Z_e^X$ so that $e_{i}\subseteq \partial U$ and $e_{i+1}\subseteq \partial V$, then there is an edge in $U\cap V$ that is dual to $Z_e^X$}. 
In the abstract carrier of $Z_e^X$, there are polygons $\bar{U}$, $\bar{V}$ containing edges $\bar{e}_i$ and $\bar{e}_{i+1}$ respectively so that the natural map carries $\bar{U},\bar{V},\bar{e}_i,\bar{e}_{i+1}$ to $U,V,e_i,e_{i+1}$ respectively. 
Since the abstract carrier embeds (recall Proposition~\ref{projected hypercarrier embed}), $\bar{e}_i$ and $\bar{e}_{i+1}$ must intersect in a vertex because their images $e_i,e_{i+1}$ share a vertex. 
By construction of the abstract carrier, any two polygons that have nontrivial intersection in the abstract carrier of $Z_e^X$ share an edge dual to $Z_e^X$. This proves the claim. 

We may assume that $U_1,U_2,\ldots,U_n$ are listed in the sequence they appear by  traversing $\gamma_e^X$ from $e$ to $e'$ (recalling that $\gamma_e^X$ is a subpath of a tree that does not backtrack). Next we claim that if $U_i\cap U_j \ne \emptyset$, then $j\in \{i-1,i,i+1\}$. For $i<n$ $U_i,U_{i+1}$ intersect in an edge dual to $Z_e^X$ and if $j\ne i,i+1$, $U_j$ cannot contain the same dual edge or $\gamma_e^X$ would not be a subpath of an embedded tree. 
Similarly, for $j\ne i,i-1$ and $i>1$, $U_j$ cannot contain the edge of $\gamma_e^X$ dual to edges of $U_i,U_{i-1}$. Then the preceding claim implies that among $U_1,\ldots,U_n$, $U_i$ may only intersect itself and $U_{i\pm 1}$. 

Thus a path from $U_1$ to $U_n$ must intersect every $U_i$ for $1\le i\le n-1$. 
\end{proof}

\begin{lemma}\label{L: bad wall polygon}   
Let $\gamma$ be a geodesic in $X\oskel$. 
Let $\gamma_e^X$ be an embedded geodesic path in $Z_e^X$ between edges $e$ and another distinct edge $e'$ on $\gamma$ and let $U$ be a polygon containing $e$ where $\gamma_e^X$ is dual to $e$ and a second edge $\hat{e}$ of $\partial U$. Then
\begin{enumerate}
\item \label{ineq: bad wall large intersection}
the edge $\hat{e}$ satisfies $d_{X_{bal}\oskel}(\hat{e}, \gamma \cap \boundary U) < \frac16 |\boundary U| - 1$, and 
\item \label{ineq: geodesic carrier polygon}
$\frac13 |\partial U| < |\gamma\cap \partial U| \le \frac12 |\partial U|$. 
\end{enumerate}
\end{lemma}

\begin{proof}
By \Cref{L: bad walls follow geodesics} the path $\gamma_e^X$ lies entirely in polygons whose boundary intersects $\gamma$. Therefore, $\gamma_e^X$ must exit $U$ at $\hat{e}$  through a piece of $\boundary U$ that intersects $\gamma$, so the first claim follows immediately by the $C'(\frac16)$ condition. 

Since $\gamma$ is a geodesic, we have the upper bound $|\gamma \cap 
\partial U|\le \frac12|\partial U|$. 
For the lower bound in \Cref{ineq: geodesic carrier polygon}, $e$ lies on $\gamma$ and $e,\hat{e}$ are diametrically opposed, so $d_{X_{bal}\oskel}(e, e') \geq \frac12|\boundary U| - 1$.  
With the first claim, 
\[
|\gamma\cap \boundary U| \geq  |e| + d_{X_{bal}\oskel}(e, \hat{e}) - d_{X_{bal}\oskel}(\hat{e}, \gamma\cap \boundary U) > 1+ \tfrac12 |\boundary U| - (\tfrac16|\boundary U| - 1) > \tfrac13|\partial U|.
\qedhere
\] 
\end{proof}

\begin{proof}[Proof of \Cref{P: onecross}]
Let $f$ be an edge of $\gamma$.
If $Z_f^X$ intersects $\gamma$ exactly once, we are done.
Thus, suppose $f' \subset \gamma$ is an edge such that $f \neq f'$ and $Z_f^X$ intersects $f'$ nontrivially.
Since $f$ and $f'$ are distinct, the definition of a hypergraph implies there exists a polygon $U$ in the carrier of $Z_f^X$ with $f \subset \boundary U$.   
We will use $\gamma \cap \boundary U$ to detect a hypergraph crossing $\gamma$ exactly once.  

Order edges of $\gamma$ and label them consecutively by $e_i$ for $i \in \Z$ such that $\gamma \cap \boundary U = e_1 \cup \dots \cup e_k$.  
By \Cref{L: bad wall polygon},  we have that  $\frac13 |\boundary U| \le k \leq \frac12 |\boundary U|$.  
Hence, we may choose $i$ between $1$ and $k$ such that
\begin{equation}
\tag{$\star$}
\label{ineq: middle edge}
\tfrac16 |\boundary U| - 1 < \min \{ d_{X_{bal}\oskel}(e_i, e_1) \, , \, d_{X_{bal}\oskel}(e_i, e_k). \}
\end{equation}
In particular, every piece of $\boundary U$ containing $e_1$ or $e_k$ is disjoint from $e_i$ by the $C'(\frac16)$ condition.
Let $e := e_i$ for the above choice of $i$ and let $Z_e^X$ be the hypergraph through $e$.
We will see that $Z_e^X$ only crosses $\gamma$ at a single point (the midpoint of $e$).

Assume for contradiction that there exists an edge $e' \subset \gamma$ with $e' \neq e$  and let $\gamma_e^X \subseteq Z_e^X$ be a geodesic segment joining the midpoints of $e$ and $e'$.  
Let $\check{U}$ denote the polygon with $e\subseteq \partial\check U$ so that the interior of $\check U$ has nontrivial intersection with $\gamma_e^X$.  
We now prove $\check{U} \neq U$.   
Indeed, let $\hat{e} \subset \boundary U$ be the edge opposite to $e$ in $\boundary U$.  
Clearly, $\hat{e} \cap \gamma = \varnothing$, and the inequality in \Cref{ineq: middle edge} immediately implies the same inequality with $e$ replaced by $\hat{e}$.  
Hence, $\gamma_e^X$ cannot cross $\hat{e}$ as this would contradict \Cref{L: bad wall polygon}(\ref{ineq: bad wall large intersection}).  

We specially label one more edge of $\gamma$ to complete this proof.
In the labelling of $\gamma$ let $e_\ell := e'$. 
Either $\ell < 1$ or $\ell > k$ by convexity of hypercarriers.  
If $\ell < 1$ fix $j = 0$, otherwise, fix $j = k + 1$.  
By choice of the labelling, $e_j$ is the edge closest to $e$ on the subinterval of $\gamma$ between $e$ and $e'$ whose interior is disjoint from $\gamma \cap \boundary U$.  
There exists a polygon $W \neq U$ that is adjacent to $e_j$ and whose interior has nontrivial intersection with $\gamma_e^X$ by \Cref{L: bad walls follow geodesics}.

Note that $\check{U} \neq W$ for otherwise $\boundary \check{U}$ would contain the entire subinterval of $\gamma$ joining $e$ and $e_j$.  
Such a subinterval would yield a piece between $\boundary \check{U}$ and $\boundary U$ containing $e$ and one of $e_1$ or $e_k$, which is impossible by how we chose of $e$.  

Martin and Steenbock showed that hypergraphs in $X$ are embedded trees \cite[Theorem~3.10]{MartinSteenbock}.  
It follows that any polygon whose interior intersects $Z_e^X$ disconnects the carrier $Y_e^X$.  
In particular, $U$ and $W$ must lie in distinct components of $Y_e^X \smallsetminus \check{U}$ by construction of $\gamma_e^X$.  
However, $U$ and $W$ have nontrivial intersection containing an endpoint of $e_j$ a contradiction.  
Thus, $Z_e^X$ crosses $\gamma$ exactly once.

Since $G$ acts cocompactly on $X_{bal}$, there exists some $N\in\naturals$ so that any polygon $U$ of $X_{bal}$ has $|\partial U|\le N$. Based on the above argument, every length $2N$ subsegment of $\gamma$ must contain some edge $e$ so that $Z_e^X$ intersects $\gamma$ exactly once. 
\end{proof}

\section{Relatively cubulating blowups of polygonal complexes}	
\label{S: blowup main}

In this section we provide a general framework for establishing relatively geometric actions for refined peripheral structures of a given relatively hyperbolic groups that acts on a polygonal complex whose 1--skeleton is a fine hyperbolic graph.  

\subsection{Relatively Geometric Blowups}

\begin{definition}\label{D: blowup}
Let $(K,\mc{D})$ be a relatively hyperbolic pair acting on a polygonal complex $X_0$ satisfying Hypotheses~\ref{H: polygonal wall hypotheses}. 
A \textbf{relatively geometric blowup} of $(X_0,K,\mc{D})$ is a cell complex $\EX$ with a $K$--action and a $K$--equivariant surjective combinatorial projection $\pi:\EX \to X_0$. so that:
\begin{enumerate}
\item \label{I: simply connected blow up}
$\EX$ is simply connected and the cells are either $n$--cubes or even-sided polygons. 
\item \label{I: cocompact blow up}
$K$ acts cocompactly by isometries on $\EX$. 
\item For any $x \in X_0$, the $\pi$--fiber is a (connected) CAT(0) cube complex, and if $\pi\inv(x)$ contains more than one point then $x$ is a vertex. \label{I: blowup fiber comp from vertices} 
\item every cell not in the $\pi$--fiber of a vertex is an edge or an even sided polygon. \label{I: polygons lift to blowup polygons}
\item \label{I: vertex condition}
Each $D \in \mc{D}$ admits a relatively hyperbolic structure $(D, \mc{R}_D)$ so that 
if $v$ is a vertex stabilized by $D$ then $(D, \mc{R}_D)$ acts relatively geometrically on $\pi\inv(v)$.
\end{enumerate}
If $v$ is a vertex, we call $\pi\inv(v)$ a \textbf{fiber complex}. 
\end{definition}

The blowup in \Cref{D: blowup} is a modification of that introduced by Martin and Steenbock \cite{MartinSteenbock} who replace condition (\ref{I: vertex condition}) with $D$ acting geometrically on a CAT(0) cube complex.
Observe that the equivariance of the projection ensures that for $g\in K$, $D^g$ acts relatively geometrically on the fiber complex over the vertex stabilized by $D^g$ with respect to $\mc{R}_D^g = \{R^g:\,R\in \mc{R}_D\}$. 

\begin{remark}
\label{R: trivial blowup}
Note that $X_0$ with the identity map is a relatively geometric blowup of itself where the peripheral structure on a stabilizer $K_v$ of a vertex $v$ is $(K_v,\{K_v\})$, which acts relatively geometrically on $v$ by the trivial action. 
\end{remark}

\begin{notation}
Given $D\in \mc{D}$ we write $ED$ for the fiber complex over the vertex stabilized by $D$. 
\end{notation}

Suppose that $(\EX,\pi)$ is a relatively geometric blowup of $(X_0,K,\mc{D})$. Then each $D\in \mc{D}$ stabilizes a vertex of $X_0$ and has a peripheral structure $\mc{R}_D$. 
If we let $\mc{R} = \bigcup_{D\in \mc{D}}\mc{R}_D$, then $(K,\mc{R})$ is a relatively hyperbolic pair. 
We call $(K,\mc{R})$ the \textbf{refined peripheral structure for the blowup}. 

The following is a straightforward consequence of the fact that the action of $K$ on $X_0$ witnesses the relative hyperbolicity of $(K,\mc{R})$ and \cite[Theorem 1.1]{WenyuanYang2014}:
\begin{proposition}
If $v$ is a vertex of $X_0$, then $\Stab_K(v)$ is relatively quasiconvex in $(K,\mc{R})$.
\end{proposition}

Thus if $D\in \mc{D}$, we obtain an inclusion $\partial_{\mc{R}_D} D\hookrightarrow \partial_{\mc{R}} K$ whose image is the limit set of $D$. 

We recall a useful result that will be used for fullness:
\begin{proposition}[{\cite[Proposition 6.1]{GenFine}}]\label{P: hyperplane parabolic iff sub-graph intersects}
Let $W$ be a hyperplane of $ED$, the fiber complex stabilized by $D$ and let $R_{ED}$ be a maximal parabolic subgroup in $(D,\mc{R}_D)$. 
$W$ is dual to an edge whose stabilizer is commensurable to $R_{ED}$ if and only if $R_{ED}$ is commensurable to a subgroup of $\Stab_D (W)$. 
\end{proposition}

\subsection{Separating points in the Bowditch Boundary of a fiber complex stabilizer}

Recall that our goal is to use Theorem~\ref{t:Rel BW} to produce a relatively geometric action by finding full relatively quasiconvex subgroups that separate any pair of points in the Bowditch Boundary. The next result shows that Theorem~\ref{T: hypergraph separation} is enough provided that we can separate pairs of points in the boundary of a single fiber complex stabilizer:

\begin{cor}\label{C: multifiber bdd separation}
Let $p,q\in \partial_{\mc{R}}K$ be distinct. If $p,q$ are not points in the Bowditch Boundary of a single fiber complex stabilizer, then there exists a full relatively quasiconvex subgroup $H_{p,q}$ in $(K,\mc{R})$ so that $p,q$ are in $H_{p,q}$--distinct components of $\partial_{\mc{R}}G\setminus \Lambda H_{p,q}$.
\end{cor}

\begin{proof}
Recall from \Cref{P: peripheral boundary map} that there exists a continuous map $\phi:\partial_{\mc{R}}G\to\partial_{\mc{D}}G$ and $\phi(p)\ne\phi(q)$ if they are not in the boundary of a single fiber group. \Cref{T: hypergraph separation} and \Cref{P: projection preimages separation} now imply the desired result.
\end{proof}

\cite[Theorem 1.4]{GenFine} implies that any two distinct points in the Bowditch Boundary of a group that acts relatively geometrically can be separated into distinct complementary components of the limit set of a hyperplane stabilizer and a finite index subgroup of the hyperplane stabilizer does not identify these components up to group action. 
Every hyperplane in a fiber complex is the intersection of that fiber complex with an $\EX$--hyperstructure $W_e^{\EX}$.
We aim to use this fact by showing that we can extend the stabilizers of these hyperplanes to codimension-one subgroups of $K$.
We recall two useful results that were used in the process of proving \cite[Theorem 1.4]{GenFine}:
\begin{lemma}[{\cite[Lemma 6.7]{GenFine}}]\label{biinfinite sidedness}
Let $(L,\mc{F})$ act relatively geometrically on a CAT(0) cube complex $\tilde{C}$.
Let $\gamma:(-\infty,\infty)\to \tilde{C}\oskel$ be a connected bi-infinite combinatorial $(\lambda,\epsilon)$--quasi-geodesic. Given $M>1$, there exist a hyperplane $W$ of $\tilde{C}$ and $t_M>0$ so that for all $t$ with $|t|>t_M$:
\begin{enumerate}
\item $\gamma(\pm t)$ are separated by $W$,
\item $\gamma$ crosses $W$ an odd number of times, and
\item $d(\gamma(t),W)>M$.
\end{enumerate}
Also, $\gamma(t)\in W$ implies $|t|\le t_M$. 
\end{lemma}

\begin{lemma}[{\cite[Lemma 6.8]{GenFine}}]\label{ray sidedness}
Let $(L,\mc{F})$ act relatively geometrically on a CAT(0) cube complex $\tilde{C}$.
Let $\gamma:[0,\infty)\to \tilde{C}$ be an infinite combinatorial $(\lambda,\epsilon)$--quasi-geodesic ray where $x = \gamma(0)$ is a vertex. Given $M>1$, there exists a hyperplane $W$ of $\tilde{C}$ and $t_M>0$ so that for all $t$ with $t>t_M$:
\begin{enumerate}
\item $W$ separates $\gamma(0)$ and $\gamma(t)$,
\item if $\Stab_L(x)$ is infinite, $W$ is not dual to any edge whose stabilizer is commensurable to $\Stab_L(x)$, and
\item $d(\gamma(t),W)>M$. 
\end{enumerate}
In particular, if $\gamma(t)\in W$, then $0<t \le t_M$. 
\end{lemma}

Recall that the vertex groups $D^g$ where $D\in\mc{D}$ and $g\in K$ act relatively geometrically on their associated fiber complex $ED^g$.    
This gives the 1--skeleton, $(ED^g)\oskel$, the structure of a generalized fine hyperbolic graph with respect to this action.
Generalized fine graphs are easier to work with for separating boundary points as was shown in \cite{GenFine}. 
In particular, the complete electrification map described in \Cref{electricQI} is an equivariant quasi-isometry to a fine hyperbolic graph.  

\begin{proposition}
\label{EG:gen fine}
$\EX\oskel$ is a generalized fine hyperbolic graph with respect to the action of $(K,\mc{R})$. 
\end{proposition}

\begin{proof}
By \cite[Theorem 5.1]{CharneyCrisp} (a general version of \Cref{CCrestatement}), $\EX$ is equivariantly quasi-isometric to the coned-off Cayley graph for $(K,\mc{R})$. Since the actions of the fiber groups on the fiber complexes are relatively geometric, every maximal parabolic stabilizes a vertex of $\EX$. 
Likewise, if $\Sigma_{R^g}$ is the sub-graph of $\EX\oskel$ whose cells have stabilizer commensurable to $R^g$ for some $R\in\mc{R}$ and $g\in K$, then $\Sigma_{R^g}$ lies in a fiber complex and is connected and compact by \cite[Proposition 3.5]{RelGeom}. If $\Gamma$ is the complete electrification of $\EX$ with respect to $\Sigma_{R^g}$, then $\Gamma$ has trivial edge stabilizers, and the maximal parabolic subgroups of $(K,\mc{R})$ each stabilize exactly one vertex of $\Gamma$. Therefore, by \Cref{CCrestatement}, $\Gamma$ is also quasi-isometric to the coned-off Cayley graph for $(K,\mc{R})$ and is hyperbolic. Hence, by \cite[Lemma 4.5]{BowditchRH}, $\Gamma$ is a fine hyperbolic graph. 
Therefore by \cite[Proposition 3.2]{GenFine}, $\EX\oskel$ is generalized fine with respect to the action of $(K,\mc{R})$. 
\end{proof}

\begin{figure}[h]
\begin{tikzcd}
\EX \arrow[rr, "electrify"] \arrow[d, "\pi", swap]
& &	\Gamma \arrow[dll] \\
X_0 
\end{tikzcd}
\caption{Diagram relating the projection map $\pi$ with the electrification in the proof of \Cref{EG:gen fine}.}
\end{figure}

We fix the following notation for the remainder of this section.
Let $p,q\in \partial_{\mc{R}_D}(D)$ be distinct. Recall $ED$ is the fiber complex over a vertex stabilized by $D$. Then by \cite[Theorem 1.4]{GenFine}, there exists a hyperplane $W_D$ in $ED$ so that $\Lambda\Stab_D (W_D)$ separates $p$ and $q$ in $\partial_{\mc{P}_D}D$. Let $e$ be an edge dual to $W_D$. Then $W_e^{\EX}$ contains $W_D$ in $\EX$. 

Let $\beta:\EX\oskel\to\Gamma$ be the equivariant quasi-isometry that collapses the $\Sigma_{D^g}$ from \Cref{EG:gen fine} where $\Gamma$ is a fine hyperbolic graph witnessing the relative hyperbolicity of $(K, \mc{R})$ in the sense of \Cref{D: fine relhyp}. Let $L$ be the hyperset associated to $W_e^{\EX}$ in $\EX$ and let $J$ be the associated hypercarrier. 

We now introduce a few additional conditions on the blowup that will be helpful to control projections of the hyperstructures $W_e^{\EX}$:

\begin{definition}\label{D: blowup projection properties}
Let $(\EX,\pi)$ be a relatively geometric blowup of $(X_0,K,\mc{D})$ as in \Cref{D: blowup}. 
\begin{enumerate}
\setcounter{enumi}{5}
\item  \label{I:PWT}
$(\EX,\pi)$ has the \textbf{projected wall tree property} if $\pi$ projects every $\EX$--hyperstructure to an (embedded) tree in $X_0$ whose intersection with $X_0\oskel$ is quasiconvex.   
\item \label{I:PWF}
$(\EX,\pi)$ has the \textbf{projected wall fullness property} if the following holds: whenever $W_e^{\EX}$ is a $\EX-$hyperstructure and $v$ is a vertex of $X_0$ so that $\Stab_K(W_e^{\EX})$ is infinite, $\Stab_K(W_e^{\EX})$ has infinite intersection with $\Stab_K(v)$ if and only if $\pi(W_e^{\EX})$ intersects $v$. 
\item \label{I:2SWP}
$(\EX,\pi)$ has the \textbf{two-sided wall projection property} if for every $W_e^{\EX}$, $\pi$ maps the complementary components of $W_e^{\EX}$ in $\EX$ to distinct components of $\pi(W_e^{\EX})$.
\end{enumerate}
\end{definition}

Note that Item~\ref{I:2SWP} verifies that the projection of the wall space structure on the blow-up induces a wallspace structure on base polygonal complex.

\begin{proposition}\label{P: blowup wall relqc}
If $(\EX,\pi)$ has the projected wall tree property, then $\Stab_K(W_e^{\EX})$ is relatively quasiconvex in $(K,\mc{D})$.  
\end{proposition}

\begin{proof} 
$\Stab_G(W_e^{\EX})$ acts cocompactly on the quasiconvex tree $\pi(W_e^{\EX})$ by \Cref{L: hyperstructure stab cocompact}. Therefore, \Cref{T: qccrit} implies that $\Stab_G(W_e^{\EX})$ is relatively quasiconvex in $(K,\mc{D})$. 
\end{proof}

\begin{proposition}\label{P: blowup wall full}
Suppose the blow-up $(\EX,\pi)$ has the projected wall tree and fullness properties.  Let $H = \Stab_K(W_e^{\EX})$. Then:
\begin{enumerate}
\item If $ED^g$ is a fiber complex and $W_e^{\EX}$ intersects $ED^g$, then $W_e^{\EX}\cap ED^g$ is a hyperplane and $H\cap \Stab_K(v)$ is the stabilizer of a hyperplane of $ED^g$. 
\item $H$ is full relatively quasiconvex in $(K,\mc{R})$.
\end{enumerate} 
\end{proposition}

\begin{proof}
For the first claim, if $W_e^{\EX}$ intersects $ED^g$ in more than one hyperplane, then the projection $\pi(W_e^{\EX})$ will contain a loop. This contradicts the fact that $\pi(W_e^{\EX})$ is a tree.  

Let $W_{D^g} = W_e^{\EX}\cap ED^g$ so that $W_{D^g}$ is a hyperplane of $ED^g$. If $h\in H\cap D^g$ then $h\cdot W_v$ is a hyperplane of $ED^g$. Therefore, $h\cdot W_{D^g} = W_{D^g}$ because $ED^g \cap W_e^{\EX} = W_{D^g}$. Hence $H\cap D^g \le \Stab_{D^g}(W_{D^g})$. 
Similarly, if $h'\in \Stab_{D^g}(W_{D^g})$, then the intersection of $h'\cdot W_e^{\EX}$ with $ED^g$ is $W_{D^g}$. Then $h'\cdot W_e^{\EX} = W_e^{\EX}$ because $W_e^{\EX}$ is the unique $\EX--$hyperstructure that intersects $ED^g$ in $W_{D^g}$. Then $\Stab_{D^g}(W_{D^g}) \le D^g\cap H$. 

Relative quasiconvexity of $H$ in $(K,\mc{D})$ follows from \Cref{P: blowup wall relqc}. 
By \cite[Theorem 1.3]{WenyuanYang2014}, $H$ is relatively quasiconvex in $(K,\mc{R})$ exactly when $H\cap Q^g$ is relatively quasiconvex in $(K,\mc{R})$ for all $g\in G$ and $Q\in\mc{D}$. 
Since $Q^g$ stabilizes a fiber complex, $H\cap Q^g$ is a hyperplane stabilizer which is relatively quasiconvex in $(K,\mc{R})$ by \Cref{P: hypstabrelqc}. 

If $W_e^{\EX}$ intersects $ED^g$, then $D^g\cap H$ stabilizes a hyperplane of $ED^g$. Hyperplane stabilizers of relatively geometric actions are full by \Cref{hyp stab full} so for any parabolic (with respect to $(K,\mc{R})$) subgroup $R$ of $D^g$, $H\cap R$ is either finite or finite index in $R$. 

By the projected wall fullness property, $H$ intersects the stabilizer of a fiber complex associated to a vertex $v$ of $X_0$ exactly when $\pi(W_e^{\EX})$ passes through $v$. Therefore, if $W_e^{\EX}$ does not intersect the fiber complex over $v$:
\[|\Stab_K(v)\cap H| \le |\Stab_K(v)\cap \Stab_K(p(W_e^{\EX}))|<\infty.\]
Since every maximal parabolic is contained the stabilizer of a fiber complex, it follows that $H$ is full in $(K,\mc{R})$. 
\end{proof}

Since $\Stab_G (W_e^{\EX})$ is relatively quasiconvex in $(K,\mc{R})$ by \Cref{P: blowup wall relqc}, $L = W_e^{\EX}\cap \EX\oskel$ is quasiconvex in $\EX\oskel$, and hence the associated hypercarrier $J$ is quasiconvex. The following property follows immediately from the two-sided wall projection property and the fact that $\pi$ is continuous:
\begin{proposition}
If $(\EX,\pi)$ has the two-sided wall projection property, each $W_e^{EG}$ separates $\EX$ into two complementary components which naturally divides $J$ into $J^+$ and $J^-$ so that $J^+\cap J^- = L$.

Further, every path between vertices of components of $\EG_{bal}\oskel\setminus L$ must traverse an edge dual to $W_e^{\EG}$. 
\end{proposition}

By \Cref{P: collapsing hypersets}, $\beta(L)$ is a hyperset in the fine hyperbolic graph $\Gamma$ and $\beta(J)$ is its hypercarrier. 
Moreover, because we are working with the refined structure $(K, \mc{R})$ each of the stabilized sub-graphs $\Sigma_{R^g}$ for $R \in \mc{R}$ and $g\in K$ is contained in a fiber complex.  

\begin{proposition}\label{P: two-sided carrier holds}
If $(\EX,\pi)$ has the two-sided wall projection, projected wall tree and projected wall fullness property, then the hypercarrier $\beta(J)$ satisfies the two-sided carrier property (Recall \Cref{D: two-sided carrier}).
\end{proposition}

\begin{proof}
Observe that $\beta(J^+)\cup \beta(J^-) =\beta(J)$. If $z$ is a vertex of $\beta(J^+) \cap \beta(J^-)$, then $z$ is the image of some $\Sigma_{D^g}$ which lies entirely in a fiber complex $ED^g$. The only way for $\beta(J^+)$ and $\beta(J^-)$ to intersect in a vertex $v$ is if an edge $e'$ dual to $W_e^{\EX}$ contained in $\Sigma_{D^g}$ collapses. 
Let $EX_v$ be the fiber complex over $v$ and let $K_v = \Stab_K(v)$. 
Then the hyperplane $W_{e'}$ in $EX_v$ dual to an edge $e'$ of $EX_v$ is contained in $W_e^{\EX}$ and $\Stab_{K_v}(W_{e'})\le \Stab_{K} (W_e^{\EG})$ by \Cref{P: blowup wall full}. 
Since $\Sigma_{D^g}$ contains $e'$, $D^g$ is commensurable to a subgroup of $\Stab_{K_v}(W_{e'})\le \Stab_{G} (W_e^{\EG})$ by \Cref{P: hyperplane parabolic iff sub-graph intersects}. 
Thus $\Stab_K(v)$ is commensurable to a subgroup of $\Stab_K (W_e^{\EG})$ and thus $\Lambda\Stab_K(v) \subseteq \Lambda \Stab_K (W_e^{\EG})$.  

Since every path between components of $\EX\oskel\setminus L$ must traverse an edge dual to $W_e^{\EG}$, every path between components of $\Gamma\setminus \beta(L)$ must intersect $\beta(J^+)\cap \beta(J^-)$. 
\end{proof}

\begin{lemma}\label{comp components of A enough}
Let $L$ be the hyperset associated to $W_e^{\EX}$. If $D\in \mc{D}$ and $x,y \in ED\oskel\setminus L$ are distinct points that lie in distinct complementary components of $L$, then $x$ and $y$ lie in distinct complementary components of $\EX\oskel \setminus L$.  
\end{lemma}

\begin{proof}
Since $W_e^{\EG}$ splits $\EX$ into two components, any path in $\EX\oskel$ has endpoints in different components of $\EX\oskel\setminus L$ if and only if the path crosses $L$ an odd number of times. Likewise, complementary components of $ED\oskel\setminus W_e^{\EG}$ are determined by the number of times a path crosses the hyperplane $W_e^{\EX}\cap ED$. 
\end{proof}

To prove relative quasiconvexity of $\Stab_K(W_e^{\EX})$ we need relative quasiconvexity of hyperplane stabilizers of fiber complexes in $(K,\mc{R})$. 
\begin{proposition}\label{P: hypstabrelqc}
Suppose $H_W$ is a hyperplane stabilizer for a hyperplane in a fiber complex $ED^g$:
\begin{enumerate}
\item $H_W$ is relatively quasiconvex in $(D^g,\mc{P}_D^g)$
\item $H_W$ is relatively quasiconvex in $(K,\mc{R})$. 
\end{enumerate}
\end{proposition}

\begin{proof}
The first claim follows from \cite[Corollary 4.11]{GenFine}.

Since $(D^g,\mc{R}_D^g)$ is the peripheral structure induced by $(K,\mc{R})$ 
and $H_W$ is relatively quasiconvex in $(D,\mc{R}_D^g)$, $H_W$ is relatively quasiconvex in $(K,\mc{R})$ by \cite[Corollary 9.3]{Hruska2010}. 
\end{proof}

\begin{proposition}\label{P: separation in fiber case}
Let $D\in \mc{D}$. 
Let $p,q\in \partial_{\mc{R}}K$ be distinct so that $p,q\in \partial_{\mc{R}_D}D$. 
There exists an edge $e$ of $ED$ so that if $L$ is the hyperset associated to $W_e^{\EX}$, the hyperset $\beta(L)$ separates $p,q$ in the sense of \Cref{D: separation}. 
\end{proposition}

\begin{proof}
Choose $W_e^{\EX}$ as above so that if $W_D = W_e^{\EX}\cap ED$, $\Lambda\Stab_D (W_D)$ separates $p,q$ in $\partial_{\mc{R}_D}D$.  
We describe points in the Bowditch Boundary $\partial_{\mc{R}} K$ in terms of the fine hyperbolic graph $\Gamma$ as described in \Cref{D: fine boundary}.
As in the proof of \Cref{P: hypergraph separation}, the proof splits into the following three cases: (1) two parabolic points, (2) two conical limit points, and (3) one parabolic and one conical limit point. 

\textbf{Case 1: $p,q$ are both parabolic vertices of $\Gamma$.}  
If $p,q$ are both parabolic vertices, then $p,q$ already lie in distinct components of $\EX\setminus W_e^{\EX}$ because $W_e^{\EX}$ intersects $ED$ exactly in the hyperplane dual to $e$. 
Also, $p,q\notin \Lambda \Stab_K (W_e^{\EX})$ because $p,q\in \partial_{\mc{R}_D}D$ and $\Lambda\Stab_D (W_D)$ separates $p,q$ in $\partial_{\mc{R}_D}D$.  
Note that $\beta(p)$ and $\beta(q)$ cannot be in $L$ because otherwise $p$ or $q$ would have to be in $\Lambda\Stab_D(W_D)\subseteq \Lambda \Stab_K(W_e^{\EX})$ by \cite[Proposition 6.1]{GenFine}. Then $\beta(p),\beta(q)$ are in distinct components of $\Gamma \setminus \beta(L)$. 

\textbf{Setup for the remaining cases: }
Let $\gamma$ be a geodesic in $\Gamma$ between $p,q$ (recall \Cref{biinfinite-geod} implies existence of $\gamma$). 
Recall that $(K, \mc R)$ admits a generalized fine action on $\EX\oskel$ \Cref{EG:gen fine}.
By \cite[Definition 2.6\, and Corollary 2.11]{GenFine}, there exists a quasi-geodesic $\hat\gamma$ in $\EX\oskel$ called the \textbf{complete de-electrification} of $\gamma$ such that $\beta(\hat{\gamma}) = \gamma$. 
%

\textbf{Case 2: $p,q$ are both conical limit points.} 
Now suppose $p,q$ are both conical limit points.  
Since $W_D$ separates $p,q$, there exists a quasi-geodesic $\hat\sigma:(-\infty,\infty)\to ED$ so that $\sigma  = \beta(\hat{\sigma})$ is geodesic in $\beta(ED)$ between $p,q$ and $\sigma$ has the properties described in \Cref{biinfinite sidedness}. 
By quasi-geodesic stability, there exists an $R_1\ge 0$ so that $d_{Haus}(\hat\sigma,\hat\gamma)<R_1$. 

By \Cref{biinfinite sidedness}, for any $M>0$, there exists $T>>0$ so that if $|t|>T$, $d_{ED\oskel}(W_D , \hat{\sigma}(t))>M$.

Claim: for $N>0$, there exists $T_N$ so that if $|t|>T_N$, $d_{\EX\oskel}(\hat\gamma(t),W_e^{\EX})>N$.

If there exists a sequence $t_i\to\infty$ so that $d_{\EX\oskel}(\hat\gamma(t_i),W_e^{\EX})\le N$, then eventually $\beta(\hat{\gamma})=\gamma$ remains within a bounded distance of $\beta(L)$ by quasiconvexity of $\beta(J)$. 
Then as $t\to\infty$, $\gamma(t)$ converges to either $p$ or $q$, so one of these lie in $\Lambda \Stab_K(L)$. 
Suppose without loss of generality that $\gamma(t)\to p$ as $t\to\infty$.
Then $p\in \Lambda D \cap \Lambda \Stab_K(L) = \Lambda \Stab_K (W_D)$ by relative quasiconvexity which contradicts the fact that $W_D$ separates $p,q$. 

Claim: for $|t|>>0$, $\hat\gamma(\pm t)$ lie in distinct components of $\EX\oskel \setminus L$. For each $t$, there exist $t'_t\in \reals$ so that $d(\hat\sigma(t'_t),\hat\gamma(t)) <R_1$. 
Since $\hat\gamma,\hat\sigma$ are quasi-geodesics, we can ensure for $|t|>>0$ that $|t'_t|>>0$ and $d(\hat\gamma(t),W_D)>R_1$. Thus for $t>>0$, $\hat\gamma(t)$ and $\sigma(t'_t)$ lie in the same component of $\EX\setminus L$. 
Since $\hat\sigma(\pm t')$ eventually lie in distinct components of $\EX\oskel\setminus L$, for $t>>0$, $\hat\gamma(\pm t)$ also lie in distinct components of $\EX\oskel\setminus L$. 

For any $D\in \mc{D}$ and $g\in K$, the sub-graph whose cells are stabilized by $\Sigma_{D^g}$ has finite diameter. So if $\Sigma_{D^g}$ has an edge dual to $W_e^{\EX}$, then for $t>>0$, $\hat\gamma(\pm t)$ cannot lie in $\Sigma_{D^g}$ by \Cref{biinfinite sidedness}. 
Therefore, for all  $t_- <<0 $ and $t_+>>0$ $\beta(\hat\gamma(t_-)),\beta(\hat\gamma(t_+))$ lie in distinct components of $\Gamma\setminus \beta(L)$. 
Since $\beta(\hat\gamma)=\gamma$, $\beta(L)$ separates $p,q$. 

\textbf{Case 3: $p$ is a conical limit point and $q$ is a parabolic vertex. }
Since $W_D$ separates $p,q$, there exists a quasi-geodesic $\hat\sigma:(-\infty,\infty)\to ED$ so that $\sigma  = \beta(\hat{\sigma})$ is geodesic in $\beta(ED)$ between $p,q$ and $\sigma$ has the properties described in \Cref{ray sidedness}. 
In this case, parameterize $\sigma:[0,\infty)\to \Gamma$, $\hat\sigma:[0,\infty)\to \EX\oskel$, $\hat\gamma:[0,\infty)\to \EX\oskel$ so that $\gamma(0) = \sigma(0) = \beta(\hat\gamma(0))=q$. 
By \Cref{ray sidedness}, for $t>>0$, $\hat\sigma(t)$ and $\hat\sigma(0)$ lie on different sides of $W_D$ and $d(\hat{\sigma}(t),W_D)$ can be made arbitrarily large. Following an argument similar to the one in the previous case, for $t'>>0$, $\hat\gamma(t')$ can be made arbitrarily far from $W_e^{\EX}$ so that $p \notin \Lambda \Stab_K(L)$, and $\hat\gamma(t)$ lies in the same component of $\EX\oskel\setminus L$ as $\hat\sigma(t')$ for $t,t'>>0$. 
Note that $\beta(q)$ cannot lie in $\beta(L)$ by an argument similar to the one in Case 1.  
Hence for $t''>>0$, $\gamma(t'')$ and $\gamma(0) = q$ must lie in distinct components of $\Gamma\setminus\beta(L)$. 
\end{proof}

While $D$ was used in \Cref{P: separation in fiber case} to simplify notation, $D$ can be replaced by the (maximal parabolic in $(K,\mc{D})$) stabilizer of any fiber complex without issue. 
We now show that pairs of points in $\partial_{\mc{R}} K$ with the same image in $\partial_{\mc{D}} K$ can be separated by some $\EX$--hyperstructure associated to an edge in the corresponding fiber complex.

\begin{theorem}\label{T: EG Wall bdd separation}
Suppose $(\EX,\pi)$ satisfies the projected wall tree, projected wall fullness and two-sided wall projection properties.  
Let $ED^g$ be a fiber complex over $v$ and let $(D^g,\mc{R}_D^g)$ be the induced peripheral structure on the stabilizer $D^g$ of $ED^g$. 
Let $p, q \in\partial_{\mc{R}_{D}^g} D^g \subseteq \partial_{\mc{R}}K$ be distinct, then there exist $W_e^{\EX}$ for some edge $e$ of $ED^g$ so that for some finite index $H_W\le \Stab_K(W_e^{\EX})$, $p$ and $q$ are in $H_W$--distinct components of $\partial_{\mc{D}}K \setminus H_W$. 
\end{theorem}

\begin{proof}
By \Cref{P: separation in fiber case} there exists a $W_e^{\EX}$ that separates $p,q$. The action of $K$ on $\Gamma$ has trivial edge stabilizers, the action is cocompact and every maximal parabolic fixes a vertex.  Therefore, by \Cref{T: separation criterion} (using \Cref{P: two-sided carrier holds} to ensure the hypotheses hold), there exists a finite index subgroup $H_W$ of $\Stab_K(W_e^{\EX})$ so that $p$ and $q$ lie in $H_W$-distinct components of $\partial_{\mc{R}}K\setminus \Lambda H_W$. 
\end{proof}

\subsection{Applying the boundary criterion}

We now prove a more general theorem that will imply \Cref{mainthm} after proving $\EG_{bal}$ (together with the natural projection) is a relatively geometric blowup of $(X_{bal},K,\mc{D})$. 

We will use the boundary criterion for relatively geometric cubulation criterion.
\cubulationcriterion*

The following is a naive application of \Cref{t:Rel BW} and subsequent refinements result in stronger consequences on the peripheral structures.

\begin{theorem}\label{T: blowup cubulation}
If $(\EX,\pi)$ is a relatively geometric blowup of $(X_0,K,\mc{D})$ that satisfies the projected wall tree, projected wall fullness and two-sided wall projection properties (recall Definition~\ref{D: blowup projection properties}, then there is a refined peripheral structure $(K,\mc{R}')$ of $(K,\mc{D})$ and $(K,\mc{R})$ so that $(K,\mc{R}')$ acts relatively geometrically on a CAT(0) cube complex.
\end{theorem}

\begin{proof}[Proof of \Cref{T: blowup cubulation}]
Let $p,q\in \partial_{\mc{R}}K$ be distinct points in the Bowditch Boundary. Let $\mc{R}_{D}^g$ be the induced peripheral structure on a  fiber complex  stabilizer $D^g$.  
By \cite[Theorem 1.1]{WenyuanYang2014}, the fiber complex stabilizers are relatively quasiconvex in $(K,\mc{R})$.
Thus $\partial_{\mc{R}_D}D^g$ embeds in $\partial_{\mc{R}}K$. 
Recall that $(K,\mc{D})$ is a relatively hyperbolic pair, so fiber complex stabilizers have pairwise finite intersections. Thus Bowditch boundaries with respect the induced peripheral structures of distinct fiber groups embed as disjoint subspaces in $\partial_{\mc{R}}K$. 

If $p,q$ both lie in $\partial_{\mc{R}_D^g}D^g$, then \Cref{T: EG Wall bdd separation} implies that there exists a full relatively quasiconvex (in $(K,\mc{R})$  $H \le K$ so that $p,q$ lie in $H$--distinct components of $\partial_{\mc{R}}K\setminus \Lambda H$. 

On the other hand if $p,q$ are not both in the boundary of a single fiber group, then \Cref{C: multifiber bdd separation} implies that there exists a full relatively quasiconvex $H_{p,q}\le K$ so that $p,q$ lie in $H_{p,q}$--distinct components of $\partial_{\mc{R}}K\setminus \Lambda H_{p,q}$. 

Then, \Cref{t:Rel BW} implies that there exists a refined peripheral structure $(K,\mc{R}')$ for $(K,\mc{R})$ so that $(K,\mc{R}')$ acts relatively geometrically on a CAT(0) cube complex. 
\end{proof}

Let us now verify that the moreover statement of \Cref{t:Rel BW} holds, so no refinement is needed.

\begin{proposition}
\label{P:ellipticPeripherals}
Suppose that $(\EX, \pi)$ is a relatively geometric blowup of $(X_0, K, \mc{D})$ that satisfies \Cref{D: blowup projection properties} and $(K, \mc{R})$ is the associated refined peripheral structure. 
Each peripheral subgroup $R \in \mc{R}$ acts elliptically on the dual cube complex.
\end{proposition}

\begin{lemma}
\label{L:halfspace correspondence}
Suppose that $W$ is any $\EX$--hyperstructure and $H_W < \Stab_K(W)$ is the index 2 subgroup without inversion in $W$.  A choice of complementary component determines a halfspace of some sufficiently large neighborhood of $H_W$ in $K$.
\end{lemma}
\begin{proof}
By construction hyperstructures are two-sided.  
There is an equivariant quasi-isometry between $\EX$ and the coned-off Cayley graph $\Gamma(G, \mc{R})$ by Charney and Crisp \cite[Theorem~5.1]{CharneyCrisp}, so $W$--complementary components induce (coarse) halfspaces of $H_W$ each containing a neighborhood of $\Stab_K(W)$.
Moreover, these halfspaces are $H_W$--invariant since $H_W$ does not invert $W$.

To obtain coarse halfspaces in $K$, one verifies that unconing $\Gamma(G, \mc{R})$ leaves the halfspaces unchanged.  
Indeed, if $\Sigma_{R^g}$ is the subgraph will cell stabilizers commensurable to $R^g$ (as in \Cref{EG:gen fine}) is disjoint from $W$ then $R$ cannot identify large neighborhoods of $H_W$ because it has deep intersection with at most one halfspace of $H_W$. 
On the other hand, if $\Sigma_{R^g}$ intersects $W$ then by fullness $H_W \cap R^g$ is finite index in $R^g$, so a sufficiently large neighborhood of $H_W$ will be contained in a single halfspace.
\end{proof}

\begin{proof}[Proof of \Cref{P:ellipticPeripherals}]
If any $R \in \mc{R}$ is finite then it already acts elliptically since isometries of finite dimensional CAT(0) cube complexes are semisimple \cite{Haglund2008}.  Up to cubical subdividing we can and will assume that elliptic subgroups fix a vertex.  

By definition of relatively geometric blowup, up to conjugation any peripheral subgroup stabilizes a vertex $x$ contained in some fiber complex $\pi\inv(v)$.  
Let $K_x$ denote the stabilizer of $x$.  

By construction, each hyperstructure (in either $X_0$ or $\EX$) is disjoint from vertices, so its complementary components determine a bipartition of the vertices.  

For each $X_0-$hyperstructure, choose the complementary component containing $v$.
For each $\EX-$hyperstructure choose the complementary component that contains $x$.
The above choices determine a principal ultrafilter and hence a vertex of the dual cube complex.
Indeed, the cube complex obtained using \Cref{t:Rel BW} is dual to stabilizers of hyperstructures in $X_0$ and $\EX$ and \Cref{L:halfspace correspondence} demonstrates that choices of halfspaces are determined by choices of complementary components of the corresponding hyperstructure.  

Finally, $K_x$ stabilizes $x$ in $\EX$ and hence $v$ in $X_0$.  Thus, no element of $K_x$ will flip any of the halfspaces. 
Therefore, $K_x$ fixes a vertex in the dual cube complex.
\end{proof}

\begin{cor}
\label{C: blowup refined}
Suppose that $(\EX, \pi)$ is a relatively geometric blowup of $(X_0, K, \mc{D})$ that satisfies \Cref{D: blowup projection properties} and $(K, \mc{R})$ is the associated refined peripheral structure. 
Then $K$ admits a relatively geometric action on a CAT(0) cube complex with respect to either $(K, \mc{D})$ or $(K, \mc{R})$.
\end{cor}
\begin{proof}
The relative cubulation for $(K, \mc{R})$ is immediate from \Cref{T: blowup cubulation} and \Cref{P:ellipticPeripherals}.
For $(K, \mc{D})$, recall from \Cref{R: trivial blowup} that $X_0$ is a relatively geometric blowup of itself where the refined peripheral structure is $(K, \mc{D})$.
\end{proof}

\section{Relative Cubulations for Small Cancellation Free Products of Relatively Cubulable Groups.}\label{S: conclusion}

\subsection{$\EG_{bal}$ is a blowup}
\label{S: EG hypotheses satisfied}

We now verify that if $p:\EG_{bal}\to X_{bal}$ is the natural projection from Section~\ref{WallsFromX}, then $(\EG_{bal},p)$ is a relatively geometric blowup of $(X_{bal},K,\mc{D})$.  

We previously verified Hypotheses~\ref{H: polygonal wall hypotheses} in \Cref{S: X hypotheses satisfied}. 
We now verify Definition~\ref{D: blowup} and \Cref{D: blowup projection properties} item by item:

The cell structure and even-sided condition follow from the choice of subdivision in the definition of $\EG_{bal}$, and simply connectedness and cocompactness of the $G$ action follow from \Cref{EG cocompactness}
verifying (\ref{I: simply connected blow up}) and (\ref{I: cocompact blow up}).

Items~\eqref{I: blowup fiber comp from vertices} and \eqref{I: polygons lift to blowup polygons} follow  from the description of the fibers of $\EG$ given in \Cref{SS:blowupConstruction} and the fact that the subdivision of $X,\EG$ to form $X_{bal},\EG_{bal}$ takes polygons to polygons and cubes to cubes. 

Item~\eqref{I: vertex condition} is also immediate from the construction in \Cref{SS:blowupConstruction} in that the fibers over the vertices stabilized by $A,B$ are specifically constructed to be copies of $EA$ and $EB$ admitting relatively geometric actions of $(A,\mc{P}_A)$ and $(B,\mc{P}_B)$ respectively. 

We show that $(\EG_{bal},p)$  satisfies each Items~(\ref{I:PWT}),(\ref{I:PWF}),(\ref{I:2SWP}) of \Cref{D: blowup projection properties}. 

\begin{proposition}\label{P: EG walls have projection properties}
The blowup $(\EG_{bal},p)$ has the wall tree projection property, the projected wall fullness property and two-sided wall projection property.
\end{proposition}

\begin{proof}
The wall tree projection property follows from \cite[Lemma 3.32]{MartinSteenbock}. The projected wall fullness property follows from Proposition~\ref{P: protofull}. 

To see that the two-sided wall projection property hold consider $U$ and $U^*$ the two closed halfspaces of an arbitrary wall $W = W_e^{\EG}$.  
It suffices to show that $p(U) \cap p(U^*) = p(U \cap U^*)$. 

Let $x \in p(U) \cap p(U^*)$.  
By \Cref{D: blowup}, the projection $p: \EG_{bal} \to X_{bal}$ restricts to a homeomorphism on the preimage of each open 1--cell or 2--cell of $X_{bal}$, moreover, the preimage of each 0--cell is a (connected) fiber complex.  

Hence, if $x$ is contained in the interior of a 1--cell or 2--cell then there is a unique point $\tilde{x} \in p^{-1}(x)$ and it must be that $\tilde{x} \in U \cap U^*$.  
Otherwise, $x$ is a 0--cell of $X_{bal}$ with preimage $E_x$ a fiber complex.  
Since $x \in p(U)$ it must be that $E_x \cap U \neq \varnothing$, and similarly, $E_x \cap U^* \neq \varnothing$.
Since $U,U^*$ are half spaces of $W$ that both intersect $E_x$, $U\cap E_x, \, U^*\cap E_x$ are half spaces of a hyperplane $W\cap E_x$ in $E_x$. Thus there exists a point $\tilde{x} \in E_x \cap (U \cap U^*)$ such that $p(\tilde{x}) = x \in p(U \cap U^*)$ as desired.  
\end{proof}


\subsection{The proof of \Cref{mainthm}}

We are now ready to prove \Cref{mainthm}:

\mainthm*

\begin{proof}
Construct the complexes $X_{bal}$ and $\EG_{bal}$ as constructed in \Cref{S: wall construction} with a natural combinatorial projection $pr:\EG_{bal}\to X_{bal}$. 

In Section~\ref{S: EG hypotheses satisfied}, we show that $(\EG_{bal},\pi)$ is a relatively geometric blowup of $(X_{bal},G,\mc{Q})$, which gives $G$ the peripheral structure $(G,\mc{P})$. In \Cref{P: EG walls have projection properties}, we show the remaining properties in \Cref{D: blowup projection properties} hold. 
Then \Cref{C: blowup refined} implies the result. 
\end{proof}

\bibliography{cubes}
\bibliographystyle{alpha}
\end{document}